\documentclass{amsart}
\usepackage{amsmath,amssymb,amsbsy,amscd,amsthm,mathrsfs,ulem,
mathtools,multirow,
enumitem}

\usepackage[all]{xy}

\theoremstyle{plain}
\newtheorem{thm}{Theorem}[section]
\newtheorem{lem}[thm]{Lemma}
\newtheorem{prop}[thm]{Proposition}
\newtheorem{cor}[thm]{Corollary}

\newtheorem{problem}[thm]{Problem}

\newtheorem{example}[thm]{Example}
\newtheorem{definition}[thm]{Definition}
\newtheorem{rem}[thm]{Remark}
\newtheorem{q}[thm]{Question}

\newcommand{\Ga}{{\bf G}_a}
\newcommand{\A}{{\bf A}}
\newcommand{\cA}{{\mathcal A}}
\newcommand{\cK}{{\mathcal K}}
\newcommand{\cP}{{\mathcal P}}

\newcommand{\Spec }{\mathop{\rm Spec}\nolimits}

\newcommand{\lin }{\mathop{\rm lin}\nolimits}

\newcommand{\zs}{\{ 0\} }
\newcommand{\sm}{\setminus}
\newcommand{\GL }{\mathit{GL}}

\newcommand{\Aut }{\mathop{\rm Aut}\nolimits}
\newcommand{\Aff }{\mathop{\rm Aff}\nolimits}
\newcommand{\T }{\mathop{\rm T}\nolimits}
\newcommand{\J }{\mathop{\rm J}\nolimits}
\newcommand{\id}{{\rm id}}

\newcommand{\nd}{\noindent}
\newcommand{\ol}{\overline}

\newcommand{\Q}{{\bf Q}}
\newcommand{\Z}{{\bf Z}}

\newcommand{\ep}{\epsilon}

\newcommand{\x}{{\boldsymbol x}}

\newcommand{\z}{{\boldsymbol z}}

\newcommand{\kx}{k[\x]}

\newcommand{\Rx}{R[\x]}
\newcommand{\Rxyz}{R[x,y,\z]}
\newcommand{\kxyz}{k[x,y,\z]}
\newcommand{\Rxx}{R[\x^{\pm 1}]}

\begin{document}

\title[Exponentiality of automorphisms of order $p$]
{On exponentiality of automorphisms of 
$\A ^n$ of order $p$ in characteristic $p>0$}

\author{Shigeru Kuroda}

\address{Department of Mathematical Sciences\\
Tokyo Metropolitan University\\
1-1 Minami-Osawa, Hachioji, Tokyo, 192-0397, Japan}
\email{kuroda@tmu.ac.jp}

\thanks{This work is partly supported by JSPS KAKENHI
Grant Numbers 18K03219 and 22K03273}

\subjclass[2020]{Primary 14R10, Secondary 14R20, 13A50. }


\begin{abstract}
Let $X$ be an integral affine scheme 
of characteristic $p>0$, 
and $\sigma $ a non-identity automorphism of $X$. 
If $\sigma $ is {\it exponential}, 
i.e., induced from a $\Ga $-action on $X$, 
then $\sigma $ is obviously of order $p$. 
It is easy to see that the converse is not true in general. 
In fact, 
there exists $X$ which admits an automorphism of order $p$, 
but admits no non-trivial $\Ga $-actions. 
However, 
the situation is not clear 
in the case where $X$ is the affine space $\A _R^n$, 
because $\A _R^n$ admits various $\Ga $-actions 
as well as automorphisms of order $p$.

In this paper, we study exponentiality of 
automorphisms of $\A _R^n$ of order $p$, 
where the difficulty stems from 
the non-uniqueness of $\Ga $-actions inducing an exponential automorphism. 
Our main results are as follows. 
(1) We show that the triangular automorphisms of $\A _R^n$ of order $p$ 
are exponential in some low-dimensional cases. 
(2) We construct a non-exponential automorphism of $\A _R^n$ of order $p$ 
for each $n\ge 2$. 
Here, $R$ is any UFD which is not a field. 
(3) We investigate the 
$\Ga $-actions inducing 
an elementary automorphism of $\A _R^n$. 
\end{abstract}

\maketitle

\section{Introduction}\label{sect:intro}
\setcounter{equation}{0}

Throughout this paper, 
all rings are integral domains containing fields 
of characteristic $p\ge 0$. 
Let $B$ be a ring 
and $\kappa $ the prime field. 
Then, 
an action of the additive group scheme 
$\Ga =\Spec \kappa [T]$ on $\Spec B$ 
is described as a homomorphism 
$E:B\to B[T]=B\otimes _\kappa \kappa [T]$ of rings 
with the following conditions (A1) and (A2), 
where $T$ is a variable. 

\smallskip 

\nd (A1) $E_0=\id $. 
Here, 
we define 
\begin{equation}\label{eq:E_a}
E_a:B\stackrel{E}{\to }B[T]\ni f(T)\mapsto f(a)\in B 
\end{equation}
for each element $a$ 
of the {\it invariant ring} $B^E:=\{ b\in B\mid E(b)=b\} $. 

\smallskip 

\nd (A2) 
Let $\Delta :\kappa [T]\to \kappa [T]\otimes _\kappa \kappa [T]$ 
be the homomorphism 
of $\kappa $-algebras defined by $\Delta (T)=T\otimes 1+1\otimes T$. 
Then, 
the following diagram commutes: 
$$
\xymatrix
{
B \ar[d]_E \ar[r]^(.41){E} 
& B\otimes _\kappa \kappa [T] \ar[d]^{E\otimes \id } \\ 
B\otimes _\kappa \kappa [T] \ar[r]^(.41){\id \otimes \Delta } 
& B\otimes _\kappa \kappa [T]\otimes _\kappa \kappa [T]  }
$$

For simplicity, 
we call such an $E$ a {\it $\Ga $-action on} $B$, 
which is 
{\it nontrivial} if $B^E\ne B$. 
Let $R$ be a subring of $B$. 
Then, we call a $\Ga $-action $E$ on $B$ 
a $\Ga $-action {\it over} $R$ if $R\subset B^E$, 
or equivalently, 
$E$ is a homomorphism of $R$-algebras.

We denote by $\Aut B$ (resp.\ $\Aut _RB$) 
the automorphism group of the ring $B$ 
(resp.\ $R$-algebra $B$). 
For $\sigma \in \Aut B$, 
we define $B^{\sigma }:=\{ b\in B\mid \sigma (b)=b\} $.

Let $E$ be any $\Ga $-action on $B$, 
and $a\in B^E$. 
Then, $E_a$ is an element of $\Aut _{B^E}B$ 
with inverse $E_{-a}$, 
since 
$E_{a+b}=E_a\circ E_b$ for each $a,b\in B^E$ by (A2), 
$E_0=\id $ by (A1), 
and $E(c)=c$ gives $E_a(c)=c$ for each $c\in B^E$. 
We note that $B^{E_a}=B^E$ if $p=0$, 
and $B^{E_a}$ is transcendental over $B^E$ if $p>0$, 
unless $E_a=\id $ 
(cf.~Remark~\ref{rem:local} (vi)). 
Since $(E_a)^l=E_{la}$ for all $l\in \Z $, 
we have $(E_a)^p=E_0=\id $ if $p>0$.

\begin{example}\label{ex:Ga-action on R[x]}\rm
Let $R[x]$ be the polynomial ring in one variable 
over a ring $R$. 
For $a\in R\sm \zs $, 
we define $E:R[x]\ni f(x)\mapsto f(x+aT)\in R[x][T]$. 

\nd (i) 
$E$ is a nontrivial $\Ga $-action on $R[x]$ over $R$ 
with $R[x]^E=R$. 

\nd (ii) 
For $b\in R$, 
the map $E_b$ is an element of $\Aut _RR[x]$ 
such that $E_b(x)=x+ab$. 

\nd (iii) 
If $p=0$, 
then we have $R[x]^{E_1}=R$, 
since $\deg (f(x+a)-f(x))=\deg f(x)-1\ge 0$ 
holds for any $f(x)\in R[x]\sm R$. 
See Lemma~\ref{lem:invariant ring} for the case $p>0$. 

\end{example}

\begin{definition}\label{def:exp}\rm 
We say that $\sigma \in \Aut B$ is 
{\it exponential} 
if there exists a $\Ga $-action $E$ on $B$ 
and $a\in B^E$ such that $\sigma =E_a$. 
If we can choose $E$ to be a $\Ga $-action over $R$, 
i.e., 
$R\subset B^E$, 
then we say that $\sigma $ is exponential 
{\it over} $R$. 
\end{definition}

For example, 
$\sigma \in \Aut _RR[x]$ is exponential over $R$ 
if $\sigma (x)\in x+R$ 
by Example~\ref{ex:Ga-action on R[x]}. 
It should be noted that if $\sigma \in \Aut _RB$ is exponential, 
then $\sigma $ is exponential over $R$ when $p=0$, 
but is not necessarily exponential over $R$ when $p>0$, 
since $R\subset B^{E_a}$ does not imply $R\subset B^E$ 
for a nontrivial $\Ga $-action $E$ on $B$ and $a\in B^E$.

By the following remark, 
$\sigma \in \Aut B$ is exponential (over $R$) 
if and only if $\sigma =E_1$ 
for some $\Ga $-action $E$ on $B$ 
(over $R$).

\begin{rem}\label{rem:multiple of Ga-action}\rm 
If $E$ is a $\Ga $-action on $B$ and $a\in B^E\sm \zs $, 
then 
$$
E':B\stackrel{E}{\to }B[T]\ni f(T)\mapsto f(aT)\in B[T]
$$ 
is also a $\Ga $-action on $B$ with $B^{E'}=B^E$  
and $E'_1=E_a$ (cf.~(\ref{eq:E_a})). 
\end{rem}

Now, assume that $p>0$. 
Then, 
every non-identity exponential automorphism has order $p$. 
However, the converse is not true in general. 
For example, 
let $\Rxx =R[x_1^{\pm 1},\ldots ,x_n^{\pm 1}]$ 
be the Laurent polynomial ring in $n$ variables over a ring $R$. 
If $n\ge p$, 
then $\sigma \in \Aut _R\Rxx $ defined by 
$x_1\mapsto x_2\mapsto \cdots \mapsto x_p\mapsto x_1$ 
and $x_i\mapsto x_i$ for $i>p$ has order $p$. 
This $\sigma $ is not exponential, 
because $x_1,\ldots ,x_n\in \Rxx ^*\subset \Rxx ^E$ 
for any $\Ga $-action $E$ on $\Rxx $ (cf.~Remark~\ref{rem:local} (ii)).

Let $\Rx :=R[x_1,\ldots ,x_n]$ 
be the polynomial ring in $n$ variables over a ring $R$. 
The purpose of this paper is to study the following problem when $p>0$.

\begin{problem}\label{prob:main}\rm 
Are all elements of $\Aut _R\Rx $ of order $p$ 
exponential over $R$?  
\end{problem}

In contrast to $\Rxx $, 
the ring $\Rx $ admits various $\Ga $-actions 
as well as automorphisms of order $p$, 
and the problem is quite subtle. 
The answer is known to be affirmative 
if $n=1$ or if $n=2$ and $R$ is a field 
(cf.~Remark~\ref{rem:strict triangular} (ii) 
and Theorem~\ref{thm:Osaka}), 
because in these cases the structure of $\Aut _R\Rx $ is well understood.

Problem~\ref{prob:main} relates to problems about subgroups 
of $\Aut _R\Rx $ stated below: 
Recall that $\sigma \in \Aut _R\Rx $ is said to be

\nd $\bullet $ 
{\it affine} 
if $\sigma (x_1),\ldots ,\sigma (x_n)$ are of degree one; 

\nd $\bullet $ 
{\it triangular} if 
$\sigma (x_i)\in R^*x_i+R[x_1,\ldots ,x_{i-1}]$ for $i=1,\ldots ,n$; 

\nd $\bullet $ 
{\it elementary} 
if $\sigma (x_1)-x_1\in A:=R[x_2,\ldots ,x_n]$ 
and $\sigma \in \Aut _AA[x_1]$.

Let $\Aff _n(R)$ (resp.\ $\J _n(R)$) be the subgroup of $\Aut _R\Rx $ 
consisting of all affine (resp.\ triangular) automorphisms of $\Rx $, 
and $\T _n(R)$ (resp.\ $\T _n'(R)$) that generated by 
$\Aff _n(R)
\cup \{ \sigma \in \Aut _R\Rx \mid 
\text{$\sigma $ is elementary 
(resp.\ exponential over $R$)}\} $. 
By Example~\ref{ex:Ga-action on R[x]}, 
``elementary" implies ``exponential over $R$". 
Hence, 
we have 
\begin{equation}\label{eq:TGPEGC}
\T _n(R)\subset \T '_n(R)\subset \Aut _R\Rx . 
\end{equation}
It is a difficult problem to decide if 
``$\subset $" in (\ref{eq:TGPEGC}) are ``$=$" 
(cf.~\cite[\S 2.1]{Essen}). 
A classical result due to 
Jung~\cite{Jung} and van der Kulk~\cite{Kulk} 
says that $\T _2(R)=\Aut _RR[x_1,x_2]$ if $R$ is a field. 
Nagata~\cite{Nagata} showed that 
$\T _2(R)\ne \T _2'(R)$ if $R$ is not a field. 
A breakthrough has been made by Shestakov-Umirbaev~\cite{SU} 
in 2003, 
who showed that $\T _3(R)\ne \T _3'(R)$ 
if $R$ is a field with $p=0$. 
It is an important open question 
whether 
$\T _n(R)=\Aut _R\Rx $ holds when $n\ge 3$ and $p>0$. 
The subgroup $\T _n'(R)$ is not studied well for any $p\ge 0$. 
For $p>0$, 
the author~\cite{ch order} proposed to study 
the subgroup $\T _n''(R)$ generated by 
$\Aff _n(R)
\cup \{ \sigma \in \Aut _R\Rx \mid \sigma ^p=\id \} $. 
Obviously, 
we have $\T _n'(R)\subset \T _n''(R)$, 
but we do not know whether 
$\T _n'(R)=\T _n''(R)$ or $\T _n''(R)=\Aut _R\Rx $ holds 
when $n\ge 3$, or $n=2$ and $R$ is not a field. 
To study $\T _n'(R)$ and $\T _n''(R)$, 
the automorphisms of order $p$ and 
their exponentiality are of great interest.

It should also be mentioned that 
the automorphisms of order $p$ 
have been studied from various perspectives. 
Some relevant references are 
Tanimoto~\cite{Tani,Tani2,Tani3}, 
Miyanishi~\cite{Miyanishi3}, 
Maubach~\cite{Maubach} and the author~\cite{ch order} 
(see also \cite{MI,Takeda}).

In this paper, we have three main contributions as follows.

\subsection*{1$^\circ $}\label{subsect:intro1}

We give some positive answers to 
Problem~\ref{prob:main} when $\sigma $ is triangular. 
Our results are summarized as follows.

\begin{thm}\label{thm:main}
Let $R$ be a ring with $p>0$, 
and let $\sigma \in \J _n(R)$ be of order $p$.

\nd {\rm (i)} 
If $n=2$, 
then $\sigma $ is exponential over $R$.

\nd {\rm (ii)} 
If $n=2$ and $\sigma (x_1)\ne x_1$, 
then there exist $a\in R\sm \zs $ and 
$\theta (x_1)\in \sum _{p\nmid i}Rx_1^i$ 
such that $\sigma (x_1)=x_1+a$ and 
$\sigma (x_2)=x_2+a^{-1}(\theta (x_1)-\theta (x_1+a))$.

\nd {\rm (iii)} 
If $n=3$ and $R$ is a field, 
then $\sigma $ is exponential over $R$. 
\end{thm}

We note that (iii) is a consequence of (i) 
for the following reason: 
If $\sigma \in \J _n(R)$ has order $p$, 
then $\sigma (x_1)$ is in $x_1+R$  
(cf.~Remark~\ref{rem:strict triangular} (i)). 
If moreover $\sigma (x_1)$ is in $x_1+R^*$, 
then $\sigma $ is conjugate to an elementary automorphism 
by Maubach~\cite[Lemma 3.10]{Maubach} 
(cf.~Lemma~\ref{lem:order p triangular}), 
and so exponential over $R$. 
We may regard $\sigma \in \J _n(R)$ as an element of 
$\J _{n-r}(R[x_1,\ldots ,x_r])$ 
if $\sigma (x_i)=x_i$ for $i=1,\ldots ,r$.

The automorphism described in (ii) 
is a special case of the {\it Nagata type automorphism}. 
For this type of automorphism, 
the invariant ring was studied 
in detail 
in \cite{ch order}. 
When $R$ is a UFD, 
Theorem 1.2 of \cite{ch order} asserts that 
for $\sigma $ in (ii), 
the $R$-algebra 
$R[x_1,x_2]^{\sigma }$ 
is generated by at most three elements, 
and $R[x_1,x_2]^{\sigma }\simeq _RR[x_1,x_2]$ 
if and only if the ideal $(a,d\theta (x_1)/dx_1)$ 
of $R[x_1,x_2]$ is principal.

We mention that 
Tanimoto~\cite{Tani,Tani2,Tani3} studied triangular automorphisms 
of order $p$ over a field $k$ with $p>0$. 
In \cite{Tani}, 
he described all $\sigma \in \J _3(k)$ of order $p$, 
and showed that their invariant rings are 
generated by at most four elements over $k$. 
However, 
the exponentiality of $\sigma $ is not clear from his result. 
In fact, 
in the case where 
$\sigma $ lies in $\J _2(k[x_1])\sm \J _1(k[x_1,x_2])$, 
his description of $\sigma $ essentially differs from 
that in Theorem~\ref{thm:main} (ii) 
(see (II) of Theorems 3.1 in \cite{Tani}).

\subsection*{2$^\circ $}\label{subsect:intro2}

We construct elements of $\Aut _R\Rx $ of order $p$ 
which are not exponential over $R$ 
for each $n\ge 2$. 
Here, $R$ is any UFD with $p>0$ which is not a field. 
As far as we know, 
this is the first counterexample to Problem~\ref{prob:main}.

To be more precise, 
let us introduce the following notions. 
Here, 
for a subring $A$ of $B$ and $r\in B$, 
we write $A[r]=A^{[1]}$ if $r$ is transcendental over $A$.

\begin{definition}\label{defn:translation}\rm 
Let $\sigma \in \Aut B$. 

\nd{\rm (i)} 
We call $\sigma $ a {\it translation} of $B$ if 
there exists a subring $A$ of $B$ and $r\in B$ 
such that $B=A[r]=A^{[1]}$, 
$\sigma \in \Aut _AA[r]$ and $\sigma (r)\in r+A$. 
If we can choose $A$ so that $R\subset A$, 
then we call $\sigma $ a translation of $B$ {\it over} $R$.

\nd{\rm (ii)} 
We call $\sigma $ a {\it quasi-translation} of $B$ ({\it over} $R$) 
if there exists a multiplicative set $S$ of $B^{\sigma }$ 
such that the unique extension of $\sigma $ to $B_S$ is 
a translation of $B_S$ (over $R$). 
If this is the case, 
we also call $\sigma $ an $S$-{\it quasi-translation} of $B$ 
({\it over} $R$). 
\end{definition}

We denote by $Q(R)$ the field of fractions of $R$.

\begin{example}[Generic elementary automorphism]
\label{ex:gea}\rm 
Let $\sigma \in \Aut _R\Rx $ and $k:=Q(R)$. 
If there exist $y_1,\ldots ,y_n\in \kx $ 
for which $\kx =k[y_1,\ldots ,y_n]$ 
and the extension $\widetilde{\sigma }\in \Aut _k\kx $ of $\sigma $ 
satisfies $\widetilde{\sigma }(y_1)-y_1\in A:=k[y_2,\ldots ,y_n]$ 
and $\widetilde{\sigma }\in \Aut _AA[y_1]$, 
then $\sigma $ is an ($R\sm \zs $)-quasi-translation of $\Rx $ 
over $R$. 
We call such $\sigma $ a {\it generic elementary automorphism} of $\Rx $. 
\end{example}

We remark that the following implications hold 
for $\id \ne \sigma \in \Aut B$. 
\begin{equation}\label{eq:ABC}
\begin{gathered}
\text{translation over $R$}
\ \stackrel{\text{(A)}}{\Longrightarrow}\ 
\text{exponential over $R$}
\ \stackrel{\text{(B)}}{\Longrightarrow}\ 
\text{quasi-translation over $R$},\\
\text{quasi-translation}\ \stackrel{\text{(C)}}{\Longrightarrow}\ 
\langle \sigma \rangle \simeq \Z /p\Z.
\end{gathered}
\end{equation}
Here, 
(A) is due to Example~\ref{ex:Ga-action on R[x]}, 
and (C) is obvious. 
See \S \ref{sect:Ga-action}, Remark~\ref{rem:local} (v) for (B). 
The non-exponential automorphisms of $\Rx $ of order $p$ 
given in this paper are generic elementary automorphisms, 
and hence quasi-translations. 
Thus, our result implies that the converse of (B) 
is false if $B=\Rx $ and $p>0$.

It is interesting to note that if $p=0$, 
then the converse of (B) is true for any $B$ 
(cf.~Remark~\ref{rem:quasi-translation} (i)). 
In both cases $p=0$ and $p>0$, 
we can easily find an example of $B$ and $\sigma \in \Aut B$ 
for which the converse of (C) fails 
(cf.~Remark~\ref{rem:quasi-translation} (ii)). 
However, 
we do not know if 
such an example exists 
in the case where $\sigma \in \Aut _R\Rx $ and $p>0$ 
(cf.~Question~\ref{q:converse of C}).

\subsection*{3$^\circ $}\label{subsect:intro3}

Assume that $\sigma \ne \id $ in Definition~\ref{defn:translation} (i). 
If $p=0$, 
then $A$ is uniquely determined by $\sigma $, 
since $B^{\sigma }=A$ by Example~\ref{ex:Ga-action on R[x]} (iii). 
In contrast, if $p>0$, then 
$A$ may not be uniquely determined by $\sigma $. 
For example, 
consider $\sigma \in \Aut _{R[x_2]}R[x_1,x_2]$ 
defined by $\sigma (x_1)=x_1+1$. 
Then, 
the condition of Definition~\ref{defn:translation} (i) is satisfied for 
$$
(A,r)=(R[x_2],x_1),(R[x_1^p-x_1+x_2],x_1),(R[(x_1^p+x_2)^p-x_1],x_1^p+x_2),
\ldots . 
$$
In view of Example~\ref{ex:Ga-action on R[x]}, 
this shows that there exist various $\Ga $-actions on $R[x_1,x_2]$ 
inducing $\sigma $. 
In general, 
the $\Ga $-actions inducing an exponential automorphism are not unique 
if $p>0$, 
which makes Problem~\ref{prob:main} difficult. 
The third main contribution of this paper is to 
study the $\Ga $-actions inducing 
an elementary automorphism of $\kx $, 
where $k$ is a field with $p>0$. 
We obtain a precise description 
of such $\Ga $-actions for $n=2$. 
As an application, 
we describe a necessary and sufficient condition 
for an element of $\Aut _RR[x_1,x_2]$ of order $p$ 
to be exponential over $R$, 
where $R$ is any UFD with $p>0$. 
We also consider the case where $n\ge 3$.

\bigskip

This paper is organized as follows. 
First, 
we prove Theorem~\ref{thm:main} in Section~\ref{sect:triangular}. 
Section~\ref{sect:Ga-action} is devoted to recalling 
facts about $\Ga $-actions and polynomial automorphisms. 
In Section~\ref{sect:non-exp}, 
we construct a family of non-exponential generic elementary automorphisms. 
We present the main result $3^\circ $ in 
Sections~\ref{sect:center} and \ref{sect:set of actions}. 
We conclude this paper with a list of questions.

\section*{Notation and convention} 

Unless otherwise stated, 
the characteristic $p$ is positive 
except in Sections~\ref{sect:intro} and \ref{sect:Ga-action}, 
$R$ and $B$ are {\it rings}, 
i.e., 
integral domains containing fields, 
and $T$, $x$, $y$, and $\x =x_1,\ldots ,x_n$ 
are variables. 
We sometimes assume that $R$ is a subring of $B$.

\subsection{Two-line notation}
If $y_1,\ldots ,y_n\in \Rx $ 
satisfy $R[y_1,\ldots ,y_n]=\Rx $, 
then for any $R$-algebra $B$ and $a_1,\ldots ,a_n\in B$, 
there exists a unique homomorphism $\phi :\Rx \to B$ 
of $R$-algebras 
with $\phi (y_i)=a_i$ for $i=1,\ldots ,n$. 
We write this $\phi $ as 
$$\begin{pmatrix}
y_1& \cdots & y_n\\
a_1& \cdots & a_n
\end{pmatrix}$$ 
if no confusion arises. 
We simply write $(a_1,\ldots ,a_n):=
\left(\begin{smallmatrix}
x_1& \cdots & x_n\\
a_1& \cdots & a_n
\end{smallmatrix}\right)$. 
In this notation, 
we have $\phi \psi =(\phi (g_1),\ldots ,\phi (g_n))$ 
for $\phi ,\psi =(g_1,\ldots ,g_n)\in \Aut _R\Rx $.

\begin{example}\label{ex:matrix notation}\rm
For $a\in R\sm \zs $ and $f(x_1)\in R[x_1]$, we define 
$$E=
\begin{pmatrix}
x_1 & x_2+f(x_1) \\
x_1+aT & x_2+f(x_1) 
\end{pmatrix}
:R[x_1,x_2]\to R[x_1,x_2][T]. 
$$
Then, 
we have 
$E(x_1)=x_1+aT$ and $E(x_2+f(x_1))=x_2+f(x_1)$. 
Since $E(x_2+f(x_1))=E(x_2)+E(f(x_1))=E(x_2)+f(x_1+aT)$, 
we see that 
$E(x_2)=x_2+f(x_1)-f(x_1+aT)$. 
By Example~\ref{ex:Ga-action on R[x]}, 
this $E$ is a $\Ga $-action on $R[x_1,x_2]=R[x_2+f(x_1)][x_1]$ 
over $R[x_2+f(x_1)]$ with $R[x_1,x_2]^E=R[x_2+f(x_1)]$. 
\end{example}

\subsection{Conjugate}
For $\phi ,\psi \in \Aut B$, 
we write $\phi ^{\psi }:=\psi \circ \phi \circ \psi ^{-1}$.

\subsection{Restriction and extension}\label{subsect:extension}
Let $\sigma \in \Aut B$ and let $E$ be a $\Ga $-action on $B$. 

\nd (1) Let $B'$ be a subring of $B$. 

\renewcommand{\theenumi}{\roman{enumi}}

\nd 
\begin{enumerate}[leftmargin=8mm]

\item 
If $\sigma ^{\pm 1}(B')\subset B'$, 
then $\sigma $ {\it restricts to} $B'$, 
i.e., 
$\sigma $ induces an element of $\Aut B'$. 
When $\sigma $ has finite order, 
$\sigma $ restricts to $B'$ if and only if $\sigma (B')\subset B'$.

\item 
If $E(B')\subset B'[T]$, 
then $E$ {\it restricts to} $B'$, i.e., 
$E$ induces a $\Ga $-action on $B'$. 

\end{enumerate}

\nd 
(2) Let $S$ be a multiplicative set of $B$. 

\begin{enumerate}[leftmargin=8mm]

\item 
If $S\subset B^{\sigma }$, 
then $\sigma $ uniquely extends to an element of $\Aut B_S$.

\item 
If $S\subset B^E$, 
then $E$ uniquely extends to a $\Ga $-action $\widetilde{E}$ on $B_S$ 
with 
$B_S^E:=(B_S)^{\widetilde{E}}=(B^E)_S$. 
By uniqueness, 
$\widetilde{E}_{\alpha }$ equals the extension of $E_{\alpha }$ 
to $B_S$ 
for each $\alpha \in B^E$. 

\end{enumerate}

\nd 
(3) 
We sometimes denote the restriction and extension 
of $\sigma $ (resp.~$E$) 
by the same symbol $\sigma $ (resp.~$E$) 
if no confusion arises.

\section{Triangular automorphisms of order $p$}\label{sect:triangular}
\setcounter{equation}{0}

In this section, 
we prove Theorem~\ref{thm:main}. 
First, we recall some basic facts. 
Let $\J _n^\circ (R)$ be the set of $\sigma \in \J _n(R)$ 
such that $\sigma (x_i)\in x_i+R[x_1,\ldots ,x_{i-1}]$ for $i=1,\ldots ,n$.

\begin{rem}\label{rem:strict triangular}\rm 
(i) If $\sigma \in \J _n(R)$ has order $p$, 
then $\sigma $ belongs to $\J _n^\circ (R)$. 
Indeed, 
if $\sigma (x_i)\in a_ix_i+R[x_1,\ldots ,x_{i-1}]$ 
with $a_i\in R^*$, 
then $\sigma ^p=\id $ gives $a_i^p=1$, 
so $a_i=1$.

\nd (ii) Observe that $\Aut _RR[x_1]=\J _1(R)$. 
Hence, if $\sigma \in \Aut _RR[x_1]$ has order $p$, 
then $\sigma $ is in $\J _1^{\circ }(R)$ by (i). 
This implies that $\sigma $ is exponential over $R$ 
by Example~\ref{ex:Ga-action on R[x]}.

\nd (iii) 
Let $a\in R\sm \zs $ and $\phi =(f_1,\ldots ,f_n)\in \J _n(R)$. 
Then, 
\begin{equation}\label{eq:triangular conjugate}
\sigma :=\begin{pmatrix}
f_1&f_2&\cdots & f_n\\
f_1+a&f_2&\cdots &f_n
\end{pmatrix}=(x_1+a,x_2,\ldots ,x_n)^\phi 
\in \Aut _R\Rx 
\end{equation}
is triangular and has order $p$. 
Moreover, 
$\sigma $ is exponential over $R$, 
since so is $(x_1+a,x_2,\ldots ,x_n)$ 
by Example~\ref{ex:Ga-action on R[x]}. 
\end{rem}

The following lemma 
is due to Maubach~\cite[Lemma 3.10]{Maubach}.

\begin{lem}[Maubach]\label{lem:order p triangular}
Let $\sigma \in \J _n(R)$ be of order $p$. 
If $a:=\sigma (x_1)-x_1$ is in $R^*$, 
then there exists 
$\phi =(x_1,f_2,\ldots ,f_n)\in \J ^\circ _n(R)$ 
such that 
$\sigma =(x_1+a,x_2,\ldots ,x_n)^{\phi }$. 
\end{lem}

The following lemma is well known 
(cf.\ e.g., \cite[Lemma 2.1]{ch order}). 
 
\begin{lem}\label{lem:invariant ring}
For $a\in R\sm \zs $, 
we define $\tau \in \Aut _RR[x]$ by $\tau (x)=x+a$. 
Then, we have 
$R[x]^{\tau }:=\{ f\in R[x]\mid \tau (f)=f\} =R[x^p-a^{p-1}x]$. 
\end{lem}

\begin{rem}\label{rem:invariant ring}\rm 
For each $f\in R[x]$ and $a\in R$, 
there exists $h\in f+R[x^p-a^{p-1}x]$ of the form 
$h=\sum _{p\nmid i}c_ix^i$ 
with $c_i\in R$, 
since $x^p-a^{p-1}x$ is a monic polynomial of degree $p$. 
\end{rem}

Now, we prove Theorem~\ref{thm:main} (i). 
First, we sketch the idea of the proof. 
In view of Remark~\ref{rem:strict triangular} (i) and (ii), 
we may assume that $a:=\sigma (x_1)-x_1\in R\sm \zs $. 
Set $R_a:=R[1/a]$ and regard $\sigma \in \J _2(R)\subset \J _2(R_a)$. 
Then, 
by Lemma~\ref{lem:order p triangular}, 
there exists $f\in R_a[x_1]$ such that 
$\sigma =(x_1+a,x_2)^{(x_1,x_2+f)}$. 
Now, 
pick any $g\in R_a[x_1^p-a^{p-1}x_1]$. 
Since $\sigma (g)=g$ by Lemma~\ref{lem:invariant ring}, 
we notice that 
\begin{equation}\label{eq:notice}
\sigma 
=\begin{pmatrix}
x_1   & x_2+f \\
x_1+a & x_2+f 
\end{pmatrix}
=\begin{pmatrix}
x_1   & x_2+f+g \\
x_1+a & x_2+f+g 
\end{pmatrix}. 
\end{equation}
Hence, 
the $\Ga $-action 
$$
E^g:=\begin{pmatrix}
x_1   & x_2+f+g \\
x_1+aT & x_2+f+g 
\end{pmatrix}:R_a[x_1,x_2]\to R_a[x_1,x_2][T]
$$
satisfies $E^g_1=\sigma $. 
We emphasize that $E^g$ may not restrict to $R[x_1,x_2]$, 
and that $E^g$ may change if we replace $g$. 
Our plan is to 
find $g$ such that $E^g$ restricts to $R[x_1,x_2]$, 
which proves that $\sigma $ is exponential on $R[x_1,x_2]$. 
We remark that $E^g(x_1)=x_1+aT\in R[x_1,x_2][T]$. 
Hence, 
$E^g$ restricts to $R[x_1,x_2]$ 
if and only if $E^g(x_2)\in R[x_1,x_2][T]$.

\begin{proof}[Proof of Theorem~{\rm \ref{thm:main}}]
(i) 
Choose $g\in R_a[x_1^p-a^{p-1}x_1]$ 
so that $f+g=\sum _{p\nmid i}c_ix_1^i$, 
where $c_i\in R_a$ (cf.~Remark~\ref{rem:invariant ring}). 
We show that $E^g(x_2)\in R[x_1,x_2][T]$. 
By Example~\ref{ex:matrix notation}, 
we can write 
$E^g(x_2)=x_2+\sum _{p\nmid i}f_i(T)$, 
where 
$$
f_i(T):=c_i\bigl( x_1^i-(x_1+aT)^i\bigr)
=-ic_iax_1^{i-1}T-\binom{i}{2}c_ia^2x_1^{i-2}T^2-\cdots -c_ia^iT^i. 
$$
Then, 
the following statements hold: 
First of all, 
since $E^g_1=\sigma \in \J _2(R)$, we have 

\nd 
(a) $x_2+\sum _{p\nmid i}f_i(1)=E_1^g(x_2)=\sigma (x_2)
\in R[x_1,x_2]$. 

Since $a\in R$, 
the following implications hold for each $i$: 

\nd 
(b) 
$c_ia\in R\Rightarrow c_ia^l\in R$ 
for all $l\ge 1\Rightarrow f_i(T)\in R[x_1,T]$. 

By (b), it suffices to show that $c_ia\in R$ for all $i$. 
Suppose that $c_ia\not\in R$ for some $i$, 
and set $m:=\max \{ i\mid c_ia\not\in R\} $. 
Then, 
we have 

\nd 
(c) $i>m\Rightarrow c_ia\in R\Rightarrow f_i(T)\in R[x_1,T]
\Rightarrow f_i(1)\in R[x_1]$. 

\nd 
(d) The coefficient $-mc_ma$ of $x_1^{m-1}$ in $f_m(1)$ 
does not belong to $R$, since $p\nmid m$. 

\nd 
(e) $i<m\Rightarrow \deg f_i(1)<m-1\Rightarrow $ 
the monomial $x_1^{m-1}$ does not appear in $f_i(1)$. 

From (c), (d) and (e), 
we see that 
the coefficient of $x_1^{m-1}$ in $\sum _{p\nmid i}f_i(1)$ 
does not belong to $R$. 
This contradicts (a). 

(ii) Since $\sigma (x_1)\ne x_1$, 
we have $a:=\sigma (x_1)-x_1\in R\sm \zs $ 
by Remark~\ref{rem:strict triangular} (i). 
Then, in the notation above, 
$\theta (x_1):=\sum _{p\nmid i}ac_ix_1^i$ 
belongs to $\sum _{p\nmid i}Rx_1^i$. 
Moreover, 
we have 
$\sigma (x_2)=x_2+\sum _{p\nmid i}f_i(1)
=x_2+a^{-1}(\theta (x_1)-\theta (x_1+a))$. 
\end{proof}

We can generalize the construction 
in the proof of Theorem~\ref{thm:main} (i) 
to the case $n\ge 3$, 
but it does not give the desired result: 
For simplicity, let $n=3$. 
Pick $\sigma \in \J _3(R)$ with $\sigma ^p=\id $ 
and $a:=\sigma (x_1)-x_1\ne 0$. 
Then, 
by Lemma~\ref{lem:order p triangular}, 
there exist $\lambda \in R_a[x_1]$ and $\mu \in R_a[x_1,x_2]$ such that 
$\sigma =\begin{psmallmatrix*}
x_1   & x_2+\lambda & x_3+\mu \\
x_1+a & x_2+\lambda & x_3+\mu \\
\end{psmallmatrix*}$ 
holds in $\J _3(R_a)$. 
As above, 
we may choose $\lambda $ from $\sum _{p\nmid i}R_ax_1^i$. 
Since 
$\mu \in R_a[x_1,x_2]=R_a[x_2+\lambda ][x_1]$ and 
$R_a[x_2+\lambda ][x_1]^\sigma 
=R_a[x_2+\lambda ][x_1^p-a^{p-1}x_1]$ 
by Lemma~\ref{lem:invariant ring}, 
we may also assume that 
$\mu \in \sum _{p\nmid i}R_a[x_2+\lambda ]x_1^i$ 
by replacing $\mu $ with an element of 
$\mu +R_a[x_2+\lambda ][x_1^p-a^{p-1}x_1]$. 
In this situation, 
the $\Ga $-action 
$E:=\begin{psmallmatrix*}
x_1   & x_2+\lambda & x_3+\mu \\
x_1+aT & x_2+\lambda & x_3+\mu \\
\end{psmallmatrix*}$ 
on $R_a[\x ]$ satisfies $E(x_2)\in \Rx [T]$ as shown above. 
However, 
$E(x_3)$ does not always belong to $\Rx [T]$.

In other words, 
there exist 
$a\in R\sm \zs $, 
$\lambda \in \sum _{p\nmid i}R_ax_1^i$ 
and $\mu \in \sum _{p\nmid i}R_a[x_2+\lambda ]x_1^i$ 
for which the $\Ga $-action 
$E:=\begin{psmallmatrix*}
x_1   & x_2+\lambda & x_3+\mu \\
x_1+aT & x_2+\lambda & x_3+\mu \\
\end{psmallmatrix*}$ on $R_a[\x ]$ 
satisfies the following:

(i) $E_1$ restricts to $\Rx $, 
so $E_1\in \J _3(R)$ 
by Remark~\ref{rem:strict triangular} (iii). 

(ii) $E(x_3)\not\in \Rx [T]$.

\nd 
We construct such $a$, $\lambda $ and $\mu $ below.


\begin{example}\label{ex:triangular}\rm
Assume that $R$ is not a field. 
Pick any $a\in R\sm (R^*\cup \zs )$, 
and set $\lambda :=a^{-1}x_1^{p+1}$, 
$\mu :=a^{-1}(x_1^{p+1}x_2^p-x_1^{p^2+1}x_2)$ 
and $f_2:=x_2+\lambda $. 
Note that if $p\ge 3$, 
then 
$\mu =a^{-1}\bigl(x_1^{p+1}(f_2^p-\lambda ^p)
-\underline{x_1^{p^2+1}}(f_2-\underline{\lambda })\bigr) $ 
lies in 
$\sum _{p\nmid i}R_a[f_2]x_1^i$.

To see (i) and (ii) above, 
it suffices to verify the following:

\quad 
(iii) $E(x_2)\in \Rx [T]$. \quad 
(iv) $E(x_3)\in a^{-1}x_1^{1+p+p^2}(T^p-T)+\Rx [T]$.

\nd 
Note that $E(\lambda )=a^{-1}(x_1^p+a^pT^p)(x_1+aT)$. 
Hence, 
the equation $E(f_2)=f_2$ gives 
$$
E(x_2)=x_2+\lambda -E(\lambda )
=x_2-x_1^pT-a^{p-1}x_1T^p-a^pT^{p+1}\in \Rx [T], 
$$
proving (iii). 
Similarly, we have 
$E(x_3)=x_3+\mu -E(\mu )=x_3+a^{-1}(q_1+q_2)$, 
where 
\begin{align*}
q_1&:=x_1^{p+1}x_2^p-E(x_1^{p+1}x_2^p) \\
&=x_1^{p+1}x_2^p
-(\underline{x_1^p}+a^pT^p)(\underline{x_1}+aT)
(x_2^p-\underline{x_1^{p^2}T^p}-a^{(p-1)p}x_1^pT^{p^2}-a^{p^2}T^{(p+1)p}),\\
q_2&:=E(x_1^{p^2+1}x_2)-x_1^{p^2+1}x_2 \\
&=(\underline{x_1^{p^2}}+a^{p^2}T^{p^2})(\underline{x_1}+aT)
(x_2-\underline{x_1^pT}-a^{p-1}x_1T^p-a^pT^{p+1})
-x_1^{p^2+1}x_2. 
\end{align*}
Since $q_1\equiv x_1^{1+p+p^2}T^p$ 
and $q_2\equiv -x_1^{1+p+p^2}T$ modulo $(a)$, 
we see that (iv) holds true.

If $p=2$, then $\mu ':=\mu -a^{-2}x_1^8$ is in 
$\sum _{p\nmid i}R_a[f_2]x_1^i$. 
Set $E':=\begin{psmallmatrix*}
x_1   & f_2 & x_3+\mu ' \\
x_1+aT & f_2 & x_3+\mu ' \\
\end{psmallmatrix*}$. 
Then, 
we have $E'(x_i)=E(x_i)$ 
for $i=1,2$, 
and 
$$
E'(x_3)=x_3+\mu '-E(\mu ')
=E(x_3)-a^{-2}(x_1^8-(x_1+aT)^8)\in E(x_3)+\Rx [T].
$$
Therefore, 
we obtain the same conclusion as above for $E'$. 

\end{example}

We do not know whether $E_1,E_1'\in \J _3(R)$ 
in Example~\ref{ex:triangular} are exponential.

\section{$\Ga $-actions and polynomial automorphisms}\label{sect:Ga-action}
\setcounter{equation}{0}

In this section, 
we assume that $p\ge 0$ unless otherwise specified. 
We recall some facts about $\Ga $-actions and polynomial automorphisms.

\begin{lem}\label{lem:R[x]}
For a $\Ga $-action $E$ on $R[x]$ over $R$, 
we set $\lambda (T):=\lambda (x,T):=E(x)-x\in R[x][T]$. 
Then, 
$\lambda (T)$ belongs to $R[T]$, 
and is {\rm additive}, i.e., 
$\lambda (T_1+T_2)=\lambda (T_1)+\lambda (T_2)$ 
holds for variables $T_1$ and $T_2$. 
\end{lem}
\begin{proof}
Note that 
(i) $\lambda (x,T)\in TR[x][T]$, 
since $\lambda (x,0)=0$ by (A1), 
and (ii) 
$\lambda (x,T_1+T_2)=\lambda (x,T_1)+\lambda (x+\lambda (x,T_1),T_2)$ 
by (A2). 
If $\lambda (x,T)$ is not in $R[T]$, 
then we have $\lambda (x+\lambda (x,T_1),T_2)
\in T_2R[x+\lambda (x,T_1)][T_2]\sm R[T_2]$ by (i). 
Hence, 
the total degree of $\lambda (x+\lambda (x,T_1),T_2)$ 
in $T_1$ and $T_2$ 
is greater than 
$\deg _T\lambda (x,T)$. 
This contradicts (ii), 
so $\lambda (x,T)$ belongs to $R[T]$. 
Then, (ii) implies that $\lambda (T)$ is additive. 
\end{proof}

The following theorem is well known 
(cf.\ e.g., \cite[\S 1.4]{MiyanishiTata}).

\begin{thm}\label{prop:local}
For every nontrivial $\Ga $-action $E$ on a ring $B$, 
there exist $r\in B$ 
and $s\in B^E\sm \zs $ 
such that $B_S=B_S^E[r]=(B_S^E)^{[1]}$, 
where 
$S:=\{ s^i\mid i\ge 0\} $. 
\end{thm}

\begin{definition}\label{defn:fcac}\rm 
Let $A$ be a subring of $B$. 
We say that $A$ is 

\nd 
(f) {\it factorially closed} in $B$ 
if $a,b\in B$ and $ab\in A\sm \zs $ imply $a,b\in A$; 

\nd 
(a) {\it algebraically closed} in $B$ 
if $a\in B$, $f(x)\in A[x]\sm \zs $ and $f(a)=0$ imply $a\in A$. 

It is easy to see that (f) implies (a) 
and $B^*\subset A$. 
\end{definition}

\begin{rem}\label{rem:local}\rm 
Let $E$, $r$ and $S$ be as in Theorem~\ref{prop:local}.

\nd (i) 
$B_S^E$ is factorially closed in $B_S=B_S^E[r]=(B_S^E)^{[1]}$. 
Hence, $B^E=B_S^E\cap B$ is factorially closed in $B$.

\nd (ii) 
$B^*$ is contained in $B^E$ by (i). 
Hence, 
for any field $k\subset B$, 
we have $k\subset B^E$, 
i.e., 
$E$ is a $\Ga $-action over $k$.

\nd (iii) 
$B^E$ is algebraically closed in $B$ by (i). 
Hence, 
$B^E[t]=(B^E)^{[1]}$ holds for any $t\in B\sm B^E$.

\nd (iv) 
Since $E$ extends to a $\Ga $-action on $B_S=B_S^E[r]$ over $B_S^E$, 
we know by Lemma~\ref{lem:R[x]} that 
$\lambda (T):=E(r)-r$ belongs to $B_S^E[T]$, 
and is additive. 
Moreover, 
since $r$ is in $B$, 
we have $\lambda (T)\in B[T]\cap B_S^E[T]=B^E[T]$.

\nd (v) 
Pick any $\alpha \in B^E$. 
Then, 
we have $E_{\alpha }(r)-r=\lambda (\alpha )\in B^E$ by (iv). 
Since $E_{\alpha }$ extends to an element of $\Aut _{B_S^E}B_S^E[r]$, 
it follows that 
$E_{\alpha }$ is a quasi-translation of $B$ over $B^E$. 
Clearly, 
$\lambda (\alpha )\ne 0$ if and only if $E_{\alpha }\ne \id $.

\nd (vi) In (v), 
assume that $E_{\alpha }\ne \id $. 
If $p=0$, then 
$B_S^E[r]^{E_{\alpha }}=B_S^E$ holds 
by Example~\ref{ex:Ga-action on R[x]}~(iii). 
Hence, we have 
$B^{E_{\alpha }}=B_S^E[r]^{E_{\alpha }}\cap B=B_S^E\cap B=B^E$. 
If $p>0$, 
then $B^{E_{\alpha }}$ is transcendental over $B^E$, 
since so is $r^p-\lambda (\alpha )^{p-1}r\in B^{E_{\alpha }} \sm B_S^E$. 
\end{rem}

Let us mention some properties of quasi-translations. 

\begin{lem}\label{lem:quasi-translation}\rm 
We define $\tau :=E_1\in \Aut _RR[x]$ 
for the $\Ga $-action $E$ in Example~\ref{ex:Ga-action on R[x]}. 
Then, 
the following {\rm (i)} and {\rm (ii)} 
hold for any subring $B$ of $R[x]$.

\nd{\rm (i)} 
$B^*$ is contained in $R[x]^\tau \cap B$. 

\nd{\rm (ii)} 
If $p=0$ and $\tau (B)\subset B$, 
then $\tau $ restricts to an element $\sigma $ of $\Aut B$. 
Moreover, $\sigma $ is exponential. 
\end{lem}
\begin{proof}
(i) We have 
$B^*\subset R[x]^*=R^*\subset R\subset R[x]^\tau $ and $B^*\subset B$.

(ii) 
It suffices to show that $E$ restricts to $B$. 
Pick any $b=f(x)\in B$. 
Write $f(x+y)=\sum _{i=0}^df_i(x)y^i$, 
where $d\ge 0$ and $f_i(x)\in R[x]$. 
Then, 
we have 
$\tau ^l(b)=f(x+la)=\sum _{i=0}^df_i(x)a^il^i$ 
for each $l\in \Z $. 
Since $p=0$, 
we can write 
$(\tau ^l(b))_{l=0}^d=(f_i(x)a^i)_{i=0}^dP$, 
where 
$P\in \GL _{d+1}(\Q )$ is a Vandermonde matrix. 
Moreover, 
$(\tau ^l(b))_{l=0}^d$ belongs to $B^{d+1}$, 
since $\tau (B)\subset B$ by assumption. 
Thus, 
$(f_i(x)a^i)_{i=0}^d$ belongs to $B^{d+1}$. 
Therefore, 
$E(b)=f(x+aT)=\sum _{i=0}^df_i(x)a^iT^i$ belongs to $B[T]$. 
\end{proof}

Lemma~\ref{lem:quasi-translation} implies 
(i) of the following remark, 
since a quasi-translation is a restriction of a translation.

\begin{rem}\label{rem:quasi-translation}\rm 
(i) 
Let $\sigma $ be a quasi-translation of $B$. 
Then, 
we have $B^*\subset B^{\sigma }$. 
If $p=0$, 
then $\sigma $ is exponential.

\nd (ii) 
If $\sigma $ is a quasi-translation of $\Rxx $, 
then 
$x_1,\ldots ,x_n$ are in $\Rxx ^*\subset \Rxx ^{\sigma }$ 
by (i). 
Hence, 
any $\sigma \in \Aut _R\Rxx $ with $\langle \sigma \rangle \simeq \Z /p\Z $ 
is an example for which the converse of (C) in (\ref{eq:ABC}) does not hold. 
\end{rem}

In the rest of this section, 
let $k$ be a field. 
For a $k$-subalgebra $A$ of $\kx $, we define 
$$
\gamma (A):=\max \{ 0\le N\le n\mid 
\exists \phi \in \Aut _k\kx \text{ s.t.\ }
\phi (k[x_1,\ldots ,x_N])\subset A\} . 
$$
The {\it rank} of a $\Ga $-action $E$ on $\kx $ 
is defined to be $n-\gamma (\kx ^E)$.

\begin{rem}\label{rem:rank 1}\rm 
If $A$ is algebraically closed in $\kx $ 
and $\phi (k[x_1,\ldots ,x_{n-1}])\subsetneq A$ 
for some $\phi \in \Aut _k\kx $, 
then we have $A=\kx $. 
Hence, $E$ has rank one if and only if 
$\kx ^E=\phi (k[x_1,\ldots ,x_{n-1}])$ 
for some $\phi \in \Aut _k\kx $ 
because of Remark~\ref{rem:local} (iii). 
\end{rem}

\begin{rem}\label{rem:rank ineq}\rm 
For $f\in \kx $, 
let $f^{\lin }\in \sum _{i=1}^nkx_i$ 
denote the linear part of $f$. 

\nd (i) 
If $(f_1,\ldots ,f_n)$ is an element of $\Aut _k\kx $, 
then $f_1^{\lin },\ldots ,f_n^{\lin }$ 
are linearly independent over $k$, 
since the Jacobian $\det (\partial f_i/\partial x_j)_{i,j}$ belongs to $k^*$. 

\nd (ii) 
For a $k$-subalgebra $A$ of $\kx $, 
we define $A^{\lin }$ to be the $k$-vector subspace of 
$\sum _{i=1}^nkx_i$ generated by $\{ f^{\lin }\mid f\in A\} $. 
Then, we have $\gamma (A)\le \dim _kA^{\lin }$ by (i). 
\end{rem}

The following result is due to Rentschler~\cite{Rentschler} when $p=0$ 
and Miyanishi~\cite{MiyanishiNagoya} when $p>0$ 
(see \cite{ML,EH,Kojima,Nihonkai} for alternative proofs).

\begin{thm}[Rentschler~\cite{Rentschler}, 
Miyanishi~\cite{MiyanishiNagoya}]\label{thm:RM}
Every nontrivial $\Ga $-action 
on $k[x_1,x_2]$ has rank one. 
\end{thm}

We note that 
\cite{Nihonkai} derived Theorem~\ref{thm:RM} from 
the following well-known theorem by simple arguments.

\begin{thm}[Jung~\cite{Jung}, van der Kulk~\cite{Kulk}]\label{thm:JvdK}
$\Aut _kk[x_1,x_2]=\T _2(k)$. 
\end{thm}

It is also known that $\Aut _kk[x_1,x_2]$ is the 
amalgamated free product of 
$G:=\Aff _2(k)$ and $J:=\J _2(k)$ 
over $G\cap J$ (cf.\ \cite[Theorem 3.3]{Nagata}). 
This implies that, 
if $\sigma \in \Aut _kk[x_1,x_2]$ has finite order, 
then there exists $\phi \in \Aut _kk[x_1,x_2]$ such that 
$\sigma ^\phi \in G\cup J$ 
(cf.~Serre~\cite{tree}; 
see also \cite{Ig} and \cite[Proposition 1.11]{Wright}). 
Using this fact,  
several researchers recently proved the following theorem\footnote{
The author announced this theorem 
on the occasion of the 13th meeting 
of Affine Algebraic Geometry at Osaka 
on March 5, 2015 (see \cite{Miyanishi3}).} 
(cf.~\cite{Maubach} and \cite{Miyanishi3}).

\begin{thm}\label{thm:Osaka}
Assume that $p>0$ and 
$\sigma \in \Aut _kk[x_1,x_2]$ has order $p$. 
Then, 
there exists $(y_1,y_2)\in \Aut _kk[x_1,x_2]$ 
such that $\sigma (y_1)\in y_1+k[y_2]$ and $\sigma (y_2)=y_2$. 
\end{thm}

Since a conjugate of an 
elementary automorphism is exponential, 
Theorem~\ref{thm:Osaka} implies that 
Problem~\ref{prob:main} has an affirmative answer 
if $n=2$ and $R$ is a field. 
Theorem~\ref{thm:Osaka} also implies that, 
for any ring $R$ with $p>0$, 
all elements of $\Aut _RR[x_1,x_2]$ of order $p$ 
are generic elementary automorphisms.

\section{Non-exponential automorphisms of order $p$}\label{sect:non-exp}
\setcounter{equation}{0}

The goal of this section is to 
construct non-exponential 
generic elementary automorphisms of $\Rx $. 
For this purpose, 
we first give a criterion for non-exponentiality.

\subsection{Modification Lemma}\label{sect:modification}

In this subsection, 
we present a technique for modifying a given $\Ga $-action 
to a simpler one. 
For this purpose, 
we need the following lemma, 
which is a variant of Gauss's lemma.

\begin{lem}\label{lem:Gauss}
Let $R$ be a UFD, 
and let $f(T),g(T)\in R[T]\sm \zs $ with $g(0)=0$. 
If $f(T)$ and $g(T)$ are primitive, 
then $f(g(T))$ is also primitive. 
\end{lem}
\begin{proof}
Suppose that $f(g(T))$ is not primitive. 
Then, 
there exists a prime $q\in R$ such that $f(g(T))\in qR[T]$. 
This implies $\ol{f}(\ol{g}(T))=0$, 
where 
$\ol{f}(T)$ and $\ol{g}(T)$ are the images of 
$f(T)$ and $g(T)$ in $(R/qR)[T]$. 
Since $R/qR$ is algebraically closed in $(R/qR)[T]$, 
this implies that $\ol{f}(T)=0$ or $\ol{g}(T)\in R/qR$. 
Since $\ol{g}(0)=0$, 
the latter implies $\ol{g}(T)=0$. 
Thus, $f(T)$ or $g(T)$ belongs to $qR[T]$, 
a contradiction. 
\end{proof}

Let $E$ be a nontrivial $\Ga $-action on $B$. 
We define $\cP _E$ to be the set of pairs $(S,r)$, 
where $S$ is a multiplicative set of $B^E$ and $r\in B_S$, 
such that $B_S=B_S^E[r]$. 
Then, 
$\cP _E$ is nonempty by Theorem~\ref{prop:local}.

Now, 
pick any $(S,r)\in \cP _E$. 
Let $\widetilde{E}$ be the extension of $E$ to $B_S$. 
Then, 
we have $\widetilde{E}(r)\ne r$, 
since $E$ is nontrivial. 
This implies that 
$B_S^E[r]=(B_S^E)^{[1]}$ by Remark~\ref{rem:local} (iii). 
Since $\widetilde{E}$ is a homomorphism of $B_S^E$-algebras, 
we have 
\begin{equation}\label{eq:modification E}
\widetilde{E}:B_S^E[r]\ni h(r)\mapsto h(r+\lambda (T))\in B_S^E[r][T], 
\end{equation}
where $\lambda (T):=\widetilde{E}(r)-r\ne 0$. 
By Lemma~\ref{lem:R[x]}, 
$\lambda (T)$ belongs to $B_S^E[T]$, 
and is additive.

With this notation, the following lemma holds. 

\begin{lem}\label{prop:Gauss}
\nd{\rm (i)} 
For any $\alpha \in B^E$, 
we have 
$\widetilde{E}_\alpha (r)-r=\lambda (\alpha )\in B_S^E$. 
Moreover, 
\begin{equation}\label{eq:E'}
E':B^E_S[r]\ni h(r)\mapsto h(r+\lambda (\alpha )T)\in B^E_S[r][T]
\end{equation}
is a $\Ga $-action on $B_S=B^E_S[r]$ with $E_1'=\widetilde{E}_{\alpha }$.

\nd{\rm (ii)} 
If $B$ is a UFD, then the $\Ga $-action 
$E'$ in $(\ref{eq:E'})$ restricts to $B$. 
\end{lem}
\begin{proof}
\nd{\rm (i)} 
The assertion is clear 
from Example~\ref{ex:Ga-action on R[x]}, 
(\ref{eq:modification E}) and (\ref{eq:E'}).

(ii) 
Since $B$ is a UFD and $0\ne \lambda (T)\in B_S^E[T]\subset B_S[T]$, 
we can write $\lambda (T)=c\lambda _0(T)$, 
where $c\in B_S$, 
and $\lambda _0(T)\in B[T]$ is primitive. 
Then, $c$ belongs to $B^E_S$, 
since $B^E_S$ is factorially closed in $B_S=B_S^E[r]$. 
Hence, we can define a $\Ga $-action $E''$ by 
$$
E'':B^E_S[r]\ni h(r)\mapsto h(r+cT)\in B^E_S[r][T]. 
$$
For $v\in B_S[T]$, 
we define $\psi _v:B_S[T]\ni g(T)\mapsto g(v)\in B_S[T]$. 
Then, 
the diagram 
$$
\xymatrix@C=20mm@R=5mm{
&\ar[dl]_-{\widetilde{E}} B_S^E[r] \ar[d]^{E''} \ar[dr]^-{E'} \\
B_S^E[r][T] & \ar[l]^(0.4){\psi _{\lambda _0(T)}} B_S^E[r][T] \ar[r]_(0.45){\psi _{\lambda _0(\alpha )T}} & B_S^E[r][T] }
$$
commutes, 
since $\lambda (T)=c\lambda _0(T)$. 
In the following, 
we show that $E''$ restricts to $B$. 
Then, 
it follows that $E'$ restricts to $B$, 
since $\lambda _0(\alpha )$ is in $B$.

First, 
recall that $E$ is a $\Ga $-action on $B$, 
so $\widetilde{E}(b)$ lies in $B[T]$ for any $b\in B$. 
Now, suppose that $E''$ does not restrict to $B$. 
Pick $b\in B$ with $E''(b)\in B_S[T]\sm B[T]$. 
Write $E''(b)=c'\mu (T)$, 
where $c'\in B_S\sm B$, 
and $\mu (T)\in B[T]$ is primitive. 
Then, 
we have 
$\widetilde{E}(b)
=\psi _{\lambda _0(T)}(E''(b))=c'\mu (\lambda _0(T))$. 
Note that $\lambda _0(0)=0$, 
since $\lambda (0)=0$ by the additivity of $\lambda (T)$. 
Hence, 
$\mu (\lambda _0(T))$ is primitive by Lemma~\ref{lem:Gauss}. 
Since $c'\in B_S\sm B$, 
it follows that $c'\mu (\lambda _0(T))\not\in B[T]$. 
This contradicts $\widetilde{E}(b)\in B[T]$. 
\end{proof}

\begin{example}\label{ex:modification}\rm 
Let $R$ be a UFD with $k:=Q(R)$, 
and $E$ a $\Ga $-action on $\Rx $ over $R$. 
If the extension $\widetilde{E}$ of $E$ to $\kx $ 
has rank one, 
then by Remark~\ref{rem:rank 1}, 
there exist $y_1,\ldots ,y_n\in \kx $ such that 
$\kx =k[y_1,\ldots ,y_n]$ and 
$\kx ^{\widetilde{E}}=k[y_2,\ldots ,y_n]$. 
In this situation, 
the $\Ga $-action 
$E':=\begin{psmallmatrix}
y_1& y_2& \cdots & y_n \\
y_1+(\widetilde{E}_{\alpha }(y_1)-y_1)T& y_2& \cdots & y_n \\
\end{psmallmatrix}$ 
on $\kx $ 
restricts to $\Rx $ for any $\alpha \in \Rx ^E$ 
by Lemma~\ref{prop:Gauss} (ii) with $(S,r):=(R\sm \zs ,y_1)$. 

\end{example}

\subsection{$f$-stability}
Let $R$ be a subring of $B$. 
For $\id \ne \sigma \in \Aut _RB$, 
we define 
\begin{align*}
\cA _R(\sigma )
&:=\{ B^E\mid 
\text{$E$ is a $\Ga $-action on $B$ over $R$ 
s.t.\ $E_\alpha =\sigma $ 
for some $\alpha \in B^E$}\} \\ 
&=\{ B^E\mid 
\text{$E$ is a $\Ga $-action on $B$ over $R$ 
s.t.\ $E_1 =\sigma $}\} ,
\end{align*}
where the last equality is due to 
Remark~\ref{rem:multiple of Ga-action}. 
This is the quotient of 
the set of $\Ga $-actions 
over $R$ inducing $\sigma $ 
by an equivalence relation. 
We remark that 
$A\in \cA _R(\sigma )$ implies $A\subset B^\sigma $, 
since $A=B^E\subset B^{E_1}=B^\sigma $ 
for some $\Ga $-action $E$ on $B$.

Recall that $x$ and $y$ are variables.

\begin{example}\label{ex:def of f-stable}\rm 
Let $A$ be a ring, $f\in A\sm \zs $ 
and $R$ a subring of $A$. 
We define 
$\ep \in \Aut _AA[x]$ by $\ep (x)=x+f$. 
Then, 
$A$ belongs to $\cA _R(\ep )$, 
since the $\Ga $-action $E:A[x]\ni h(x)\mapsto h(x+fT)\in A[x][T]$ 
satisfies $E_1=\ep $ and $A[x]^E=A\supset R$ 
(cf.~Example~\ref{ex:Ga-action on R[x]}). 
\end{example}

\begin{definition}\label{def:a-rigid}\rm 
Let $A$, $f$, $R$ and $\ep $ be as in Example~\ref{ex:def of f-stable}. 
We say that $A$ is $f$-{\it stable} over $R$ 
if $\cA _R(\ep )=\{ A\} $, 
or equivalently, 
there exists no $\Ga $-action $E$ on $A[x]$ over $R$ 
with $E_1=\ep $ and $A[x]^E\ne A$. 
\end{definition}

\begin{rem}\label{rem:f-rigid unit}\rm 
$A$ is $f$-stable over $R$ if and only if $A$ is $uf$-stable over $R$ 
for any $u\in A^*$, 
since $A[ux]=A[x]$ and $\ep (ux)=ux+uf$. 
\end{rem}

\begin{example}\label{ex:a-rigid1}\rm
Let $A=R[y]$ and $f\in R\sm \zs $. 
Then, $A$ is not $f$-stable over $R$, 
since $E:=\begin{psmallmatrix}
x   &y+x^p-f^{p-1}x \\
x+fT&y+x^p-f^{p-1}x 
\end{psmallmatrix}$ 
is a $\Ga $-action on $A[x]
=R[y+x^p-f^{p-1}x][x]$ 
over $R$ 
such that 
$A[x]^E=R[y+x^p-f^{p-1}x]\ne A$ and 
$E_1=
\begin{psmallmatrix}
x   &y+x^p-f^{p-1}x \\
x+f&y+x^p-f^{p-1}x 
\end{psmallmatrix}
=\begin{psmallmatrix}
x   &y \\
x+f&y 
\end{psmallmatrix}
=\ep $ 
(cf.~(\ref{eq:notice})). 
\end{example}

Let $(A_i)_{i\in I}$ be a family of subrings of $B$ 
such that $A_i$ is factorially closed in $B$ 
for each $i\in I$. 
Then, 
$\bigcap _{i\in I}A_i$ 
is also factorially closed in $B$.  
Hence, 
for any subring $B_0$ of $B$, 
there exists a smallest subring $B_1$ of $B$ 
such that $B_0\subset B_1$ 
and $B_1$ is factorially closed in $B$. 
We call $B_1$ the {\it factorial closure} of $B_0$ in $B$. 
Note that $B^*\subset B_1$, 
and $B_1$ is algebraically closed in $B$ 
by the remark after Definition~\ref{defn:fcac}. 
If $b\in B_0\sm \zs $, 
then all factors of $b$ in $B$ belong to $B_1$.

\begin{thm}\label{thm:a-rigid}
Let $A$, $f$ and $R$ be as in Example~$\ref{ex:def of f-stable}$. 
Then, 
$A$ is $f$-stable over $R$ 
if there exists an over ring $\widetilde{A}$ of $A$ 
with the following conditions {\rm (a)} and {\rm (b)}. 

\nd {\rm (a)} 
The factorial closure $\ol{R[f]}$ 
of $R[f]$ in $\widetilde{A}$ 
is equal to $\widetilde{A}$.

\nd {\rm (b)} 
Every $\Ga $-action on $A[x]$ over $R$ extends to 
a $\Ga $-action on $\widetilde{A}[x]$. 
\end{thm}

Before proving Theorem~\ref{thm:a-rigid}, 
let us look at some examples.

\begin{example}\label{ex:g-rigid}\rm
(i) Let $f:=\alpha +\beta x_1^{i_1}\cdots x_n^{i_n}\in \Rx $, 
where $\alpha,\beta \in R$ with $\beta \ne 0$ 
and $i_1,\ldots ,i_n\ge 1$. 
Then, 
the factorial closure of $R[f]$ in $\Rx $ is equal to $\Rx $. 
Hence, 
$\Rx $ itself is $f$-stable over $R$.

\nd (ii) 
Let $f\in R[y]\sm R$. 
If $\bar{k}$ is an algebraic closure of $Q(R)$, 
then the factorial closure of $R[f]$ in $\bar{k}[y]$ 
is $\bar{k}[y]$. 
Since every $\Ga $-action $E$ on $R[y][x]$ over $R$ 
extends to the $\Ga $-action $\id _{\bar{k}}\otimes E$ 
on $\bar{k}\otimes _RR[y][x]=\bar{k}[y][x]$, 
we see that $R[y]$ is $f$-stable over $R$. 
\end{example}

For $\sigma \in \Aut B$, 
we define $I(\sigma )$ to be the ideal of $B$ 
generated by $\{ \sigma (b)-b\mid b\in B\} $.

\begin{lem}\label{lem:fixed points}
Assume that $B=R[b_1,\ldots ,b_n]$, 
where $R$ is a subring of $B$ 
and $b_1,\ldots ,b_n\in B$. 
Then, 
for each $\sigma \in \Aut _RB$, 
we have 
$I(\sigma )=\sum _{i=1}^n(\sigma (b_i)-b_i)B$. 
\end{lem}
\begin{proof}
Set $t_i:=\sigma (b_i)-b_i$ for $i=1,\ldots ,n$. 
Pick any $b=g(b_1,\ldots ,b_n)\in B$, 
where $g\in \Rx $. 
Then, 
$\sigma (b)-b=
g(b_1+t_1,\ldots ,b_n+t_n)-
g(b_1,\ldots ,b_n)$ 
belongs to $\sum _{i=1}^nt_iB$. 
This proves $I(\sigma )\subset \sum _{i=1}^nt_iB$. 
The reverse inclusion is clear. 
\end{proof}

\begin{example}\rm
For $\ep \in \Aut _AA[x]$ in Example~\ref{ex:def of f-stable}, 
we have $I(\ep )=fA[x]$. 
\end{example}

\begin{proof}[Proof of Theorem~$\ref{thm:a-rigid}$]
Pick any $\Ga $-action $E$ on $A[x]$ over $R$ with $E_1=\ep $. 
Our goal is to show that $A[x]^E=A$. 
For this purpose, 
it suffices to check that 
$A\subset A[x]^E$, 
for if $A\subsetneq A[x]^E$, 
then $A[x]$ is algebraic over $A[x]^E$, 
so $A[x]^E=A[x]$ by Remark~\ref{rem:local} (iii). 
This contradicts that $E_1=\ep \ne \id $.

Choose $r\in A[x]$ as in Theorem~\ref{prop:local}. 
Then, 
$b:=\ep (r)-r=E_1(r)-r$ belongs to $A[x]^E\sm \zs $ 
by Remark~\ref{rem:local} (v). 
By definition, 
$b$ also belongs to $I(\ep )=fA[x]$. 
Thus, we have 
$fg=b\in A[x]^E\sm \zs $ for some $g\in A[x]$. 
Since $A[x]^E$ is factorially closed in $A[x]$, 
it follows that $f\in A[x]^E$. 
Hence, we have $R[f]\subset A[x]^E$, 
since $R\subset A[x]^E$.

By (b), 
$E$ extends to a $\Ga $-action $\widetilde{E}$ on $\widetilde{A}[x]$. 
Then, 
we have $R[f]\subset A[x]^E\subset \widetilde{A}[x]^{\widetilde{E}}$, 
and so $R[f]\subset \widetilde{A}[x]^{\widetilde{E}}\cap \widetilde{A}$. 
Since $\widetilde{A}[x]^{\widetilde{E}}$ 
is factorially closed in $\widetilde{A}[x]$, 
we see that $\widetilde{A}[x]^{\widetilde{E}}\cap \widetilde{A}$ 
is factorially closed in $\widetilde{A}$. 
Thus, 
$\widetilde{A}=\ol{R[f]}\subset 
\widetilde{A}[x]^{\widetilde{E}}\cap \widetilde{A}$ 
holds by (a). 
Then, 
intersecting with $A[x]$, 
we obtain 
$A\subset A[x]\cap \widetilde{A}\subset 
A[x]\cap \widetilde{A}[x]^{\widetilde{E}}=A[x]^E$. 
\end{proof}

\subsection{Non-exponentiality criterion}
Let $B$ be a UFD, 
$R$ a subring of $B$ and $S:=R\sm \zs $. 
Pick any $S$-quasi-translation $\sigma $ of $B$. 
Then, there exists a subring $A$ of $B_S$ 
and $r\in B_S$ for which $B_S=A[r]=A^{[1]}$ 
and $\sigma $ extends to 
$\widetilde{\sigma }\in \Aut _AA[r]$ 
with $f:=\widetilde{\sigma }(r)-r\in A$: 
$$
\vcenter{
\xymatrix@R=0.2mm@C=1mm
{
A[r]\ar@{-)}[r] & h(r) \ar@{|->}[rr]^-{\widetilde{\sigma }} & & h(r+f)\ar@{(-}[r]& A[r]\\
\cup  & & & & \cup \\
B \ar[rrrr]^-{\sigma } &  & & & B
}}
$$
We note that $A$ contains $k:=Q(R)$, 
since $k\subset B_S$ and $(B_S)^*=A[r]^*=A^*$. 
We define a $\Ga $-action $\widetilde{E}$ on $B_S=A[r]$ over $A$ by 
\begin{equation}\label{eq:E in non-exp criterion}
\widetilde{E}:A[r]\ni h(r)\mapsto h(r+fT)\in A[r][T]. 
\end{equation}
Then, we have $\widetilde{E}_1=\widetilde{\sigma }$. 
Hence, 
if $\widetilde{E}$ restricts to $B$, 
then $\sigma $ is exponential over $R$.

\begin{thm}\label{thm:general criterion}
In the notation above, 
assume that $A$ is $f$-stable over $k$. 
If $\sigma $ is exponential over $R$, 
then the $\Ga $-action $\widetilde{E}$ 
in {\rm (\ref{eq:E in non-exp criterion})} 
restricts to $B$. 
\end{thm}

\begin{proof}
Note that $\cA _k(\widetilde{\sigma })=\{ A\} $ 
by the $f$-stability of $A$. 
By assumption, 
there exists 
a $\Ga $-action $F$ on $B$ over $R$ with $F_1=\sigma $. 
Since $S=R\sm \zs \subset B^F$, 
we can extend $F$ to a $\Ga $-action $\widetilde{F}$ 
on $B_S=A[r]$ over $k$. 
Then, 
$\widetilde{F}_1$ is equal to $\widetilde{\sigma }$. 
Hence, 
$B_S^F:=(B_S)^{\widetilde{F}}$ 
belongs to $\cA _k(\widetilde{\sigma })$. 
Thus, we get $A=B_S^F$. 
Then, we notice that 
(\ref{eq:E in non-exp criterion}) is the same as (\ref{eq:E'}) 
with $(E,\alpha ):=(F,1)$, 
since $f=\widetilde{\sigma }(r)-r=\widetilde{F}_1(r)-r$. 
Therefore, 
(\ref{eq:E in non-exp criterion}) restricts to $B$ 
thanks to Lemma~\ref{prop:Gauss} (ii), 
since $B$ is a UFD by assumption. 
\end{proof}

Finally, 
let $\sigma $ be a generic elementary automorphism of $\Rx $, 
which is an $(R\sm \zs )$-quasi-translation of $\Rx $. 
We define $f:=\widetilde{\sigma }(y_1)-y_1\in A=k[y_2,\ldots ,y_n]$ 
in the notation of Example~\ref{ex:gea}. 
Then, we have the following corollary to 
Theorem~\ref{thm:general criterion}.

\begin{cor}\label{cor:general criterion}
Assume that $R$ is a UFD 
and $A$ is $f$-stable over $k$. 
If the $\Ga $-action 
$\begin{psmallmatrix}
y_1  & y_2 & \cdots  & y_n  \\
y_1+fT & y_2 & \cdots  & y_n 
\end{psmallmatrix}
:A[y_1]\ni h(y_1)\mapsto h(y_1+fT)\in A[y_1][T]$
on $A[y_1]=\kx $ 
does not restrict to $\Rx $, 
then $\sigma $ is not exponential over $R$. 
\end{cor}

\subsection{Non-exponential generic elementary automorphisms}
\label{subsect:non exp aut}

Let $\Rxyz $ be the polynomial ring in $2+l$ variables over $R$, 
where $l\ge 0$, $\z :=z_1,\ldots ,z_l$, 
and $R$ is not a field. 
Pick any $u\in R\sm (R^*\cup \zs )$, 
and $d\in \Z \sm p\Z $ with $d\ge 2$. 
We define 
\begin{equation}\label{gather:c}
\begin{gathered}
 a:=p-2+(p-1)^2(d-1),\quad b:=(p+1)(d-1)+1,  \\ 
c:=pa+p(p-1)(d-1)=p(p-2)+p^2(p-1)(d-1),\\ 
\lambda :=u^{-p-1}y^{p(p-1)}+u^{-2}y^{pa+1},\quad 
\widetilde{x}:=x+\lambda \quad \text{and}\quad 
\widetilde{y}:=y+\widetilde{x}^d.  
\end{gathered}
\end{equation}
Put $(u^e):=u^e\Rxyz [T]$ for $e\ge 0$. 
Then, we note that 

\nd 
($1^\star $) 
$u^{p+1}\widetilde{x}\in y^{p(p-1)}+(u)$ and 
$u^b\widetilde{x}^{d-1}
=u(u^{p+1}\widetilde{x})^{d-1}
\in uy^{p(p-1)(d-1)}+(u^2)$;

\nd 
($2^\star $) 
$u^{(p+1)d}\widetilde{y}
=u^{(p+1)d}y+(u^{p+1}\widetilde{x})^d
\in R[x,y]$.

Now, 
pick any $g\in R[u^{(p+1)d}\widetilde{y},\z ]$. 
Then, 
$g$ belongs to $\Rxyz $ by ($2^\star $). 
We define a $\Ga $-action $E^g$ on 
$\kxyz =k[\widetilde{x},\widetilde{y},\z]$ 
over $k:=Q(R)$ by 
\begin{equation}\label{eq:non-exp}
E^g:=\begin{pmatrix*}
\widetilde{x}  & \widetilde{y} & z_1 & \cdots  & z_l  \\
\widetilde{x}+u^b(1+ug)T & \widetilde{y} & z_1 & \cdots  & z_l 
\end{pmatrix*}. 
\end{equation}

With this notation, 
the following proposition holds.

\begin{prop}\label{prop:non-exp}
$E^g_1$ restricts to $\Rxyz $, 
but $E^g$ does not restrict to $\Rxyz $. 
\end{prop}
\begin{proof}
Put $E:=E^g$. 
Since $E(z_i)=z_i$ for each $i$, 
it suffices to check the following:

\quad 
(3$^\star $) $E(y)\in \Rxyz [T]$. \quad 
(4$^\star $) $E(x)\in du^{-1}y^{c}(T-T^p)+\Rxyz [T]$.

\nd 
Since $E(\widetilde{y})=E(y)+E(\widetilde{x})^d$ 
equals $\widetilde{y}=y+\widetilde{x}^d$, 
we can write $E(y)=y+\xi $, 
where 
$\xi :=\widetilde{x}^d-E(\widetilde{x})^d
=\widetilde{x}^d-(\widetilde{x}+u^b(1+ug)T)^d$. 
By ($1^\star $), 
$\widetilde{x}^{d-i}(u^b(1+ug)T)^i$ belongs to 
$uy^{p(p-1)(d-1)}T+(u^2)$ if $i=1$, 
and to $(u^2)$ if $2\le i\le d$. 
Thus, we get 
\begin{equation}\label{eq:xi}
\xi \in -duy^{p(p-1)(d-1)}T+(u^2)\subset (u). 
\end{equation}
Therefore, $E(y)=y+\xi $ belongs to $y+(u)\subset \Rxyz [T]$, 
proving (3$^\star $).

Similarly, 
since $E(\widetilde{x})=E(x)+E(\lambda )$ 
equals $x+\lambda +u^b(1+ug)T$, 
we have 
\begin{equation}\label{eq:non-exp:E(y)}
E(x)=x+\lambda -E(\lambda )+u^b(1+ug)T\in 
\lambda -E(\lambda )+\Rxyz [T]. 
\end{equation}
We can write 
$\lambda-E(\lambda ) =
-(u^{-p-1}q_1+u^{-2}q_2)$, 
where 
$$
q_1:=(y+\xi )^{p(p-1)}-y^{p(p-1)}
=(y^p+\xi ^p)^{p-1}-y^{p(p-1)}
$$
and $q_2:=(y+\xi )^{pa+1}-y^{pa+1}$. 
From (\ref{eq:xi}) and (\ref{gather:c}), 
we see that 
\begin{align*}
q_1&\in (p-1)y^{p(p-2)}\xi ^p+\xi ^{2p}R[y^p,\xi ^p]
\subset -y^{p(p-2)}\cdot (-duy^{p(p-1)(d-1)}T)^p+(u^{2p}) \\
&\quad =du^py^{c}T^p+(u^{2p}),\\ 
q_2&\in y^{pa}\cdot (-duy^{p(p-1)(d-1)}T)+(u^2)
=-duy^{c}T+(u^2). 
\end{align*}
Thus, 
$\lambda -E(\lambda )$ is in $du^{-1}y^{c}(T-T^p)+\Rxyz [T]$. 
Then, (4$^\star $) follows from (\ref{eq:non-exp:E(y)}). 
\end{proof}

Now, 
let $\sigma ^g\in \Aut _R\Rxyz $ 
be the restriction of $E_1^g$. 
Then, 
$\sigma ^g$ is a generic elementary automorphism of $\Rxyz $, 
since 
$E_1^g\in \Aut _{k[\widetilde{y},\z ]}k[\widetilde{y},\z ][\widetilde{x}]$ 
and 
$E_1^g(\widetilde{x})-\widetilde{x}=u^b(1+ug)\in k[\widetilde{y},\z ]$. 
By Remark~\ref{rem:f-rigid unit}, 
$k[\widetilde{y},\z ]$ is $u^b(1+ug)$-stable over $k$ 
if and only if 
$k[\widetilde{y},\z ]$ is $(1+ug)$-stable over $k$. 
Since $E^g$ does not restrict to $\Rxyz $ by Proposition~\ref{prop:non-exp}, 
we get the following theorem by virtue of 
Corollary~\ref{cor:general criterion}.

\begin{thm}\label{thm:non-exp}
Assume that $R$ is a UFD\null. 
If $k[\widetilde{y},\z ]$ is $(1+ug)$-stable over $k$,  
then $\sigma ^g$ is not exponential over $R$. 
\end{thm}

For example, 
if we take $g=u^{(p+1)d}\widetilde{y}z_1\cdots z_l$, 
then $k[\widetilde{y},\z ]$ is $(1+ug)$-stable over $k$ 
by Example~\ref{ex:g-rigid} (i). 
Hence, 
$\sigma ^g$ is not exponential over $R$ by Theorem~\ref{thm:non-exp}.

\section{Centralizer of an elementary automorphism}\label{sect:center}
\setcounter{equation}{0}

We say that a generic elementary automorphism 
$\sigma $ of $\Rx $ is {\it fixed point free} 
if $f:=\widetilde{\sigma }(y_1)-y_1$ lies in $k^*$ 
in the notation of Example~\ref{ex:gea}. 
This definition does not depend on the choice of $y_1,\ldots ,y_n$ 
because of Lemma~\ref{lem:fixed points}.

If $n\ge 2$ and $\sigma $ is fixed point free, 
then $k[y_2,\ldots ,y_n]$ 
is not $f$-stable over $k$ 
by Example~\ref{ex:a-rigid1}. 
Hence, we cannot use Corollary~\ref{cor:general criterion}. 
The goal of this section is to generalize 
Corollary~\ref{cor:general criterion} 
to the case where $\sigma $ is fixed point free 
when $n=2$ (Theorem~\ref{thm:fixed point free}). 
For this purpose, 
we first study the centralizer of an elementary automorphism. 
Throughout this section, let $k$ be a field. 
For $\tau \in \Aut _k\kx $, 
we define $C(\tau ):=\{ \phi \in \Aut _k\kx \mid 
\phi \tau =\tau \phi \} $, 
the centralizer of $\tau $ in $\Aut _k\kx $.

\subsection{Centralizer}\label{subsect:centerlizer}

Assume that $n=2$. 
Pick any $t\in k^*$ and 
set $\ep  :=(x_1+t,x_2)\in \Aut _k\kx $. 
Then, 
for $(f_1,f_2)\in \Aut _k\kx $, 
we have 
\begin{equation}\label{eq:C(ep)}
(f_1,f_2)\in C(\ep )\iff \ep (f_1)=f_1+t\text{ and }\ep (f_2)=f_2. 
\end{equation}
From this, 
we see that $C(\ep )$ contains the subgroup $H(t)$ generated by 
\begin{equation}\label{eq:C(tau)generators}
(x_1+g(x_2),x_2)
\quad\text{and}\quad 
(x_1,x_2+g(x_1^p-t^{p-1}x_1))
\quad\text{for all}\quad g(x)\in xk[x]. 
\end{equation}
We also have 
$H_0:=\{ (x_1+u_1,ax_2+u_2)\mid a\in k^*,\ u_1,u_2\in k\} 
\subset C(\ep )$.

The following theorem is the 
main result of Section~\ref{subsect:centerlizer}.

\begin{thm}\label{thm:C(tau)}
In the notation above, 
we have $C(\ep )=H(t)H_0$. 
\end{thm}

For example, 
due to (\ref{eq:C(ep)}), 
any $\phi \in \Aff _2(k)\cap C(\ep )$ has the form 
\begin{equation}\label{eq:Aff cap C(ep)}
\phi =(x_1+sx_2+u_1,ax_2+u_2)=(x_1+sx_2,x_2)(x_1+u_1,ax_2+u_2)
\in H(t)H_0,
\end{equation}
where $s,u_1,u_2\in k$ and $a\in k^*$.

To prove Theorem~\ref{thm:C(tau)}, 
we need some preliminary results.

\begin{lem}\label{lem:triangular}
Let $\tau ,(f_1,\ldots ,f_n)\in \Aut _k\kx $ 
be such that $\tau (A_i)\subset A_i$ for $i=1,\ldots ,n-1$, 
where $A_i:=k[f_1,\ldots ,f_i]$. 
Then, $\tau (f_i)\in k^*f_i+A_{i-1}$ 
holds for $i=1,\ldots ,n$. 
If moreover $\tau $ has order $p$, 
then 
$\tau (f_i)\in f_i+A_{i-1}$ 
holds for $i=1,\ldots ,n$. 
\end{lem}
\begin{proof}
We claim that $\tau (A_i)=A_i$ for $i=1,\ldots ,n$, 
since $\tau (A_i)$ and $A_i$ are algebraically closed in $\kx $ 
and have the same transcendence degree over $k$. 
Hence, we have 
$A_{i-1}[\tau (f_i)]
=\tau (A_{i-1})[\tau (f_i)]=\tau (A_i)
=A_i=A_{i-1}[f_i]$. 
This implies that 
$\tau (f_i)
\in A_{i-1}^*f_i+A_{i-1}=k^*f_i+A_{i-1}$. 
The last assertion is due to Remark~\ref{rem:strict triangular}~(i). 
\end{proof}

For simplicity, we put $x:=x_1$ and $y:=x_2$. 
Set $G:=\Aff _2(k)$ and $J:=\J _2(k)$. 
Then, we note that each $\alpha \in G\sm J$ 
and $\beta \in J\sm G$ can be written as 
\begin{equation}\label{eq:alpha}
\begin{aligned}
\alpha &=(ax+by+u,cx+dy+v),
\text{ where }
a,b,c,d,u,v\in k,\ 
ad-bc\ne 0,\ b\ne 0 \\  
\beta &=(ax+c,by+q(x)), 
\text{ where }
a,b\in k^*,\ c\in k,\ 
q(x)\in k[x],\ 
\deg q(x)\ge 2. 
\end{aligned}
\end{equation}

Now, pick any $\tau \in \Aut _kk[x,y]$ 
and set 
$\delta :=\tau -\id :k[x,y]\to k[x,y]$. 
Note that $\delta $ is a $k$-linear map with $\ker \delta =k[x,y]^\tau $. 
For $\phi =(f_1,f_2)\in \Aut _kk[x,y]$ and $i=1,2$, 
we define $\Delta _i(\phi ):=\deg \delta (f_i)=\deg (\tau (f_i)-f_i)$. 
Here, 
``$\deg $" denotes the total degree, 
and $\deg 0:=-\infty $. 
With this notation, 
the following lemma holds.

\begin{lem}\label{lem:Delta induction}
Let $\phi =(f_1,f_2)\in \Aut _kk[x,y]$.

\nd{\rm (i)} 
If $\Delta _1(\phi )<\Delta _2(\phi )$, 
then 
$\Delta _1(\phi \alpha )=\Delta _2(\phi )\ge \Delta _2(\phi \alpha )$ 
holds for all $\alpha \in G\sm J$.

\nd{\rm (ii)} 
Assume that $\tau $ has order $p$. 
If $\Delta _1(\phi )\ge 1$ 
and $\Delta _1(\phi )\ge \Delta _2(\phi )$, 
then 
$\Delta _1(\phi \beta )=\Delta _1(\phi )<\Delta _2(\phi \beta )$ 
holds for all $\beta \in J\sm G$.

\end{lem}
\begin{proof}
(i) 
Writing $\alpha $ as in (\ref{eq:alpha}), 
we have 
$\Delta _1(\phi \alpha )=\deg (a\delta (f_1)+b\delta (f_2))$ and 
$\Delta _2(\phi \alpha )=\deg (c\delta (f_1)+d\delta (f_2))$. 
Since $b\ne 0$, 
and $\deg \delta (f_1)=\Delta _1(\phi )<\Delta _2(\phi )=\deg \delta (f_2)$ 
by assumption, 
the assertion follows immediately.

(ii) 
Writing $\beta $ as in (\ref{eq:alpha}), 
we have $\Delta _1(\phi \beta )=\deg \delta (af_1+c)
=\deg a\delta (f_1)=\Delta _1(\phi )$.

We show that $\Delta _1(\phi )<\Delta _2(\phi \beta )
=\deg (b\delta (f_2)+\delta (q(f_1)))$. 
Set $h:=\delta (f_1)=\tau (f_1)-f_1$. 
Since $\deg \delta (f_2)=\Delta _2(\phi )\le \Delta _1(\phi )=\deg h$ 
by assumption, 
it suffices to verify that 
$\deg h<\deg \delta (q(f_1))$. 
Write $q(x+y)-q(x)=p(x,y)y$, 
where $p(x,y)\in k[x,y]$. 
Note that $\deg _yp(x,y)\ge 1$, since $\deg q(x)\ge 2$. 
In the following, 
we prove that $p(f_1,h)\not\in k$. 
Then, 
it follows that 
the total degree of 
$\delta (q(f_1))=\tau (q(f_1))-q(f_1)=q(f_1+h)-q(f_1)=p(f_1,h)h$ 
is greater than $\deg h$.

Suppose that $p(f_1,h)\in k$. 
Then, $h$ is algebraic over $k[f_1]$, 
since $\deg _yp(x,y)\ge 1$. 
Since 
$k[f_1]$ is algebraically closed in $k[x,y]=k[f_1,f_2]$, 
it follows that $h\in k[f_1]$. 
Thus, $\tau (f_1)=f_1+h$ lies in $k[f_1]$. 
Since $\tau $ has order $p$ by assumption, 
this implies 
by Lemma~\ref{lem:triangular} that $\tau (f_1)\in f_1+k$, i.e., 
$h\in k$. 
This contradicts the assumption that 
$\deg h=\Delta _1(\phi )\ge 1$. 
\end{proof}

Next, we define  
\begin{equation}\label{eq:def:V}
\begin{aligned}
V&:=\{ (f_1,f_2)\in \Aut _kk[x,y]\mid \delta (f_i)\in k\text{ for }i=1,2\} \\
&=\{ \phi \in \Aut _kk[x,y]\mid 
\Delta _i(\phi )\le 0\text{ for }i=1,2\} .
\end{aligned}
\end{equation}

\begin{prop}\label{prop:V}
Let $\tau =(x+t_1,y+t_2)\in \Aut _kk[x,y]$, 
where $(t_1,t_2)\in k^2\sm \zs $. 
Assume that $\phi \in V$ is written as 
$\phi =\phi _1\phi _2\cdots \phi _l$, 
where $l\ge 1$, $\phi _1,\phi _3,\ldots \in J\sm G$ 
and $\phi _2,\phi _4,\ldots \in G\sm J$. 
Then, 
we have $\phi _1\in V$. 
\end{prop}
\begin{proof}
Since $\phi _1(x)$ is in $k^*x+k$, 
we have $\delta (\phi _1(x))\in k^*t_1\subset k$, 
i.e., 
$\Delta _1(\phi _1)\le 0$. 
Hence, 
supposing $\phi _1\not\in V$, 
we must have $\Delta _2(\phi _1)\ge 1$, 
and so $\Delta _1(\phi _1)<\Delta _2(\phi _1)$. 
Then, 
since $\phi _2\in G\sm J$, 
we know by Lemma~\ref{lem:Delta induction} (i) that 
$\Delta _1(\phi _1\phi _2)
=\Delta _2(\phi _1)\ge \Delta _2(\phi _1\phi _2)$. 
Hence, we have 
$\Delta _1(\phi _1\phi _2)
=\Delta _2(\phi _1)\ge 1$ and 
$\Delta _1(\phi _1\phi _2)\ge \Delta _2(\phi _1\phi _2)$. 
Then, 
since $\phi _3\in J\sm G$ and $\tau $ has order $p$, 
we know by Lemma~\ref{lem:Delta induction} (ii) that 
$\Delta _1(\phi _1\phi _2\phi _3)=\Delta _1(\phi _1\phi _2)
<\Delta _2(\phi _1\phi _2\phi _3)$. 
Since $\phi _4\in G\sm J$, 
this gives   
$\Delta _1(\phi _1\cdots \phi _4)=\Delta _2(\phi _1\phi _2\phi _3)
\ge \Delta _2(\phi _1\cdots \phi _4)$ 
by Lemma~\ref{lem:Delta induction} (i). 
Iterating this, 
we obtain a sequence 
$1\le \Delta _2(\phi _1)=\Delta _1(\phi _1\phi _2)
<\Delta _2(\phi _1\phi _2\phi _3)
=\Delta _1(\phi _1\phi _2\phi _3\phi _4)
<\cdots $. 
Since $\phi =\phi _1\phi _2\cdots \phi _l$, 
this shows that $\Delta _1(\phi )$ or $\Delta _2(\phi )$ is positive, 
contradicting $\phi \in V$. 
\end{proof}

Finally, we remark the following (I), (II) and (III).

\renewcommand{\theenumi}{\Roman{enumi}}
\renewcommand{\theenumii}{\arabic{enumii}}

\begin{enumerate}[leftmargin=9mm]

\item Write $\alpha \in G\sm J$ as in (\ref{eq:alpha}). 

\begin{enumerate}[leftmargin=3mm]

\item 
If $a\ne 0$, 
then there exist $\gamma _1,\gamma _2\in G\cap J$ 
and $u_1,u_2\in k^*$ 
such that 
$\alpha \gamma _1=(u_1x+y,x)$ 
and 
$\alpha \gamma _2=(x+u_2y,y)$, 
since $b\ne 0$. 

\item 
If $a=0$, 
then there exists $\gamma \in G\cap J$ 
such that $\alpha \gamma =(y,x)$. 

\end{enumerate}

\item 
Let $\beta =(ax+c,by+q(x))\in J$, 
where $a,b\in k^*$, $c\in k$ and $q(x)\in k[x]$. 
Then, for each	 $u,v\in k$, 
we have $\beta \gamma =(x,y+b^{-1}(q(x)-ux-v))$, 
where 
$\gamma :=(a^{-1}(x-c),b^{-1}(y-ua^{-1}(x-c)-v))
\in G\cap J$.

\item 
Let $t\in k^*$ and $s\in k$. 
Then, 
by Lemma~\ref{lem:invariant ring}, 
we have 
\begin{equation}\label{eq:delta(f)=s}
W_{s,t}:=
\{ f(x)\in k[x]\mid f(x+t)-f(x)=s\} 
=k[x^p-t^{p-1}x]+t^{-1}sx. 
\end{equation}
In fact, 
$W_{s,t}=k[x^p-t^{p-1}x]+h$ holds for any $h\in W_{s,t}$. 
\end{enumerate}

\begin{proof}[Proof of Theorem~$\ref{thm:C(tau)}$]
We only show that each $\phi \in C(\ep )$ belongs to $H(t)H_0$, 
since $C(\ep )\supset H(t)H_0$ is clear.  
By (\ref{eq:Aff cap C(ep)}), 
we may assume that $\phi \not\in G$. 
Then, by Theorem~\ref{thm:JvdK} and the remark following it, 
we can write $\phi =\phi _1\phi _2\cdots \phi _l$, 
where 

$\bullet $ 
$\phi _1,\phi _3,\ldots \in J\sm G$, 
$\phi _2,\phi _4,\ldots \in G\sm J$ and $l\ge 1$, or

$\bullet $ 
$\phi _1,\phi _3,\ldots \in G\sm J$, 
$\phi _2,\phi _4,\ldots \in J\sm G$ and $l\ge 2$.

\nd 
We show that $\phi $ belongs to $H(t)H_0$ by induction on $l$. 
Here, the base of the induction is the case $\phi \in G\cap J$. 
To complete the proof, 
it suffices to prove the following $(\dag )$:

$(\dag )$ 
There exist $1\le m\le l$ 
and $\gamma \in G\cap J$ 
such that 
$\phi _1\cdots \phi _m\gamma $ belongs to $H(t)$.

\nd 
Actually, 
$\phi _1\cdots \phi _m\gamma \in H(t)\subset C(\ep )$ 
implies 
$\phi ':=(\phi _1\cdots \phi _m\gamma )^{-1}\phi \in C(\ep )$. 
Since $\phi '
=(\phi _1\cdots \phi _m\gamma )^{-1}\phi _1\phi _2\cdots \phi _l
=(\gamma ^{-1}\phi _{m+1})\phi _{m+2}\cdots \phi _l$, 
we then have $\phi '\in H(t)H_0$ by induction assumption. 
Hence, we get $\phi =(\phi _1\cdots \phi _m\gamma )\phi '\in H(t)H_0$.

We divide the proof of $(\dag )$ into the following three cases 
(a), (b) and (c). 
Here, 
we write $\phi _1$ as in (\ref{eq:alpha}) in each case: 

(a) $\phi _1\in G\sm J$ and $a\ne 0$. \quad 
(b) $\phi _1\in G\sm J$ and $a=0$. \quad 
(c) $\phi _1\in J\sm G$.

\nd 
{\bf Case (a).} 
By (1) of (I), 
there exist $\gamma \in G\cap J$ and $u_2\in k^*$ 
such that $\phi _1\gamma =(x+u_2y,y)\in H(t)$.

\nd 
{\bf Case (b).} 
Since $\phi _1\in G\sm J$, we have $l\ge 2$. 
By (2) of (I), 
there exists $\gamma _1\in G\cap J$ such that 
$\phi _1\gamma _1=(y,x)$. 
Set $\phi _2':=\gamma _1^{-1}\phi _2\in J\sm G$. 
Then, 
we have $\phi _2'(x)\in k[x]$, 
so $(\phi _1\phi _2)(x)=((y,x)\phi _2')(x)\in k[y]$. 
This implies $\phi _1\phi _2\not\in C(\ep )$ 
in view of (\ref{eq:C(ep)}). 
Hence, 
we have $l\ge 3$. 
By (II), 
there exists $\gamma _2\in G\cap J$ 
such that $\phi _2'\gamma _2=(x,y+p(x))$, 
where $p(x)\in xk[x]$. 
Set $\phi _3':=\gamma _2^{-1}\phi _3\in G\sm J$. 
Then, by (1) and (2) of (I), 
there exist $\gamma _3\in G\cap J$ and $\hat{u}\in k$ 
such that $\phi _3'\gamma _3=(\hat{u}x+y,x)$. 
Then, 
we have 
$$
\phi _1\phi _2\phi _3\gamma _3
=\phi _1\gamma _1\phi _2'\gamma _2\phi _3'\gamma _3
=(y,x)(x,y+p(x))(\hat{u}x+y,x)
=(x+p(y)+\hat{u}y,y)\in H(t). 
$$

\nd {\bf Case (c).} 
Let $V$ be as in (\ref{eq:def:V}) with $\delta :=\ep -\id $. 
Then, we have $C(\ep )\subset V$ by (\ref{eq:C(ep)}), 
so $\phi _1\phi _2\cdots \phi _l\in C(\ep )\subset V$. 
Since $\phi _1\in J\sm G$, 
it follows that $\phi _1\in V$ by Proposition~\ref{prop:V}. 
Hence, 
$s:=\delta (\phi _1(y))=\delta (by+q(x))=q(x+t)-q(x)$ 
is in $k$. 
Thus, $q(x)$ belongs to (\ref{eq:delta(f)=s}).  
Then, using (II) with $u:=st^{-1}$ and $v:=q(0)$, 
we obtain $\gamma \in G\cap J$ 
and $q_1(x)\in xk[x]$ 
such that $\phi _1\gamma =(x,y+q_1(x^p-t^{p-1}x))\in H(t)$. 
\end{proof}

The centralizer of another type of elementary automorphism 
of $k[x,y]$ is described as follows.

\begin{thm}\label{thm:C(tau):fixed points}
Let $\ep '=(x+f(y),y)\in \Aut _kk[x,y]$, 
where $f(y)\in k[y]\sm k$. 
Then, $C(\ep ')$ is equal to 
\begin{equation}\label{eq:C(tau):fixed points}
\{ (ax+g,by+c)\mid 
(a,b,c,g)\in (k^*)^2\times k\times k[y],\ 
af(y)=f(by+c)\} . 
\end{equation}
\end{thm}
\begin{proof}
It is easy to check that 
(\ref{eq:C(tau):fixed points}) is contained in $C(\ep ')$. 
For the reverse inclusion, 
pick any $\phi =(g_1,g_2)\in C(\ep ')$. 
Then, 
we have 
$\ep '=\phi \ep '\phi ^{-1}=
\begin{psmallmatrix*}
g_1 & g_2\\
g_1+f(g_2) & g_2
\end{psmallmatrix*}$. 
Hence, the $\Ga $-action 
$E:=\begin{psmallmatrix*}
g_1 & g_2\\
g_1+f(g_2)T & g_2
\end{psmallmatrix*}$ 
on $k[x,y]$ satisfies $E_1=\ep '$. 
Thus, 
$k[x,y]^E=k[g_2]$ belongs to $\cA _k(\ep ')$. 
On the other hand, 
since 
$k[y]$ is $f(y)$-stable over $k$ by Example~\ref{ex:g-rigid} (ii), 
we know that $\cA _k(\ep ')=\{ k[y] \} $. 
Therefore, we get $k[g_2]=k[y]$. 
This implies that $(g_1,g_2)=(ax+g,by+c)$ 
for some $a,b\in k^*$, $c\in k$ and $g\in k[y]$ 
by Lemma~\ref{lem:triangular} with $(f_1,f_2):=(y,x)$. 
It remains to check that $af(y)=f(by+c)$.

By the definition of $\ep '$, 
we have $\ep '(ax)=ax+af(y)$. 
We also have 
$$
\ep '=
\begin{pmatrix}
g_1 &\! g_2\\
g_1+f(g_2) &\! g_2
\end{pmatrix}
=
\begin{pmatrix}
ax+g &\! by+c\\
ax+g+f(by+c) &\! by+c
\end{pmatrix}
=\begin{pmatrix}
ax &\! y\\
ax+f(by+c) &\! y
\end{pmatrix}. 
$$
This implies that $\ep '(ax)=ax+f(by+c)$. 
Therefore, 
we get $af(y)=f(by+c)$. 
\end{proof}

\subsection{Application}\label{subsect:appl}

The following lemma establishes a relation between 
the centralizer of a fixed point free elementary automorphism 
and rank one $\Ga $-actions.

\begin{lem}\label{lem:rank 1}
For $f\in k[x_2,\ldots ,x_n]\sm \zs $, 
set $\tau :=(x_1+f,x_2,\ldots ,x_n)\in \Aut _k\kx $. 

\nd{\rm (i)} 
Let $A$ be a subring of $\kx ^{\tau }$ and $z\in \kx $. 
If $A[z]=\kx $, 
then there exists $c\in k^*$ such that 
$\tau (z)=z+cf$. 

\nd{\rm (ii)} 
For a $k$-subalgebra $A$ of $\kx $, 
the following conditions {\rm (a)} and {\rm (b)} are equivalent. 

\nd {\rm (a)} 
$A=\kx ^E$ for some rank one $\Ga $-action $E$ on $\kx $ 
with $E_1=\tau $.

\nd {\rm (b)} 
$A=\psi (k[x_2,\ldots ,x_n])$ 
for some 
$\psi \in \Aut _k\kx $ 
with 
$\tau =
\begin{psmallmatrix}
\psi (x_1)  &\psi (x_2)&\cdots &\psi (x_n)\\
\psi (x_1)+f&\psi (x_2)&\cdots &\psi (x_n)
\end{psmallmatrix}$.

\nd{\rm (iii)} 
Assume that $f$ is in $k^*$. 
Then, 
$\psi \in \Aut _k\kx $ belongs to $C(\tau )$ 
if and only if 
$\tau =
\begin{psmallmatrix}
\psi (x_1)  &\psi (x_2)&\cdots &\psi (x_n)\\
\psi (x_1)+f&\psi (x_2)&\cdots &\psi (x_n)
\end{psmallmatrix}$. 
If this is the case, 
{\rm (b)} in {\rm (ii)} is equivalent to

\nd {\rm (b$'$)} 
$A=\psi (k[x_2,\ldots ,x_n])$ 
for some $\psi \in C(\tau )$.

\end{lem}
\begin{proof}
(i) 
Since $\tau $ lies in $\Aut _AA[z]$ by assumption, 
we have $I(\tau )=(\tau (z)-z)A[z]=(\tau (z)-z)\kx $ 
by Lemma~\ref{lem:fixed points}. 
We also have $I(\tau )=(\tau (x_1)-x_1)k[x_1,\ldots ,x_n]=f\kx $. 
Hence, 
we get $(\tau (z)-z)\kx =f\kx $. 
This proves $\tau (z)-z\in k^*f$.

(ii) 
If (a) holds, then $A=k[z_2,\ldots ,z_n]$ 
for some $(z_1,\ldots ,z_n)\in \Aut _k\kx $ by Remark~\ref{rem:rank 1}. 
Since $A[z_1]=\kx $ and $A=\kx ^E\subset \kx ^{E_1}=\kx ^{\tau }$, 
we know by (i) that $\tau (z_1)=z_1+cf$ for some $c\in k^*$. 
Then, (b) holds for 
$\psi :=(c^{-1}z_1,z_2\ldots ,z_n)$, 
since $z_2,\ldots ,z_n\in A\subset \kx ^{\tau }$.

If (b) holds, then 
$E:=\begin{psmallmatrix}
\psi (x_1)  &\psi (x_2)&\cdots &\psi (x_n)\\
\psi (x_1)+fT&\psi (x_2)&\cdots &\psi (x_n)
\end{psmallmatrix}$ 
is a rank one $\Ga $-action on $\kx $ with 
$\kx ^E=k[\psi (x_2),\ldots ,\psi (x_n)]=A$ and $E_1=\tau $. 
Hence, (a) holds.

(iii) 
Note that $\psi (f)=f$ since $f$ is in $k^*$. 
Hence, we have $\psi \in C(\tau )$ 
if and only if 
$$\tau =\psi \tau \psi ^{-1}=
\begin{psmallmatrix}
\psi (x_1)  &\psi (x_2)&\cdots &\psi (x_n)\\
\psi (x_1)+\psi (f)&\psi (x_2)&\cdots &\psi (x_n)
\end{psmallmatrix}
=\begin{psmallmatrix}
\psi (x_1)  &\psi (x_2)&\cdots &\psi (x_n)\\
\psi (x_1)+f&\psi (x_2)&\cdots &\psi (x_n)
\end{psmallmatrix}.
$$ 
The last assertion is clear. 
\end{proof}

Finally, let $R$ be a UFD with $k:=Q(R)$, 
and let $\sigma \in \Aut _RR[X_1,X_2]$ be of order $p$. 
Here, we use a system of variables $X_1$, $X_2$ 
other than $x_1$, $x_2$. 
By Theorem~\ref{thm:Osaka}, 
there exists $(x_1,x_2)\in \Aut _kk[X_1,X_2]$ 
such that 
the extension $\tau \in \Aut _kk[X_1,X_2]$ of $\sigma $ 
satisfies $f:=\tau (x_1)-x_1\in k[x_2]$ 
and $\tau (x_2)=x_2$. 
If $f$ is not in $k$, 
then $k[x_2]$ is $f$-stable over $k$ by Example~\ref{ex:g-rigid} (ii). 
In this case, 
we can use Corollary~\ref{cor:general criterion}.

Assume that $f$ is in $k^*$. 
Note that $\tau =
\begin{psmallmatrix}
x_1 & x_2 \\
x_1+f & x_2
\end{psmallmatrix}
=(x_1+f,x_2)$ 
in our notation. 
Hence, 
we can define a subgroup $H(f)$ 
of $C(\tau )$ as in (\ref{eq:C(tau)generators}). 
Then, 
the main result of Section~\ref{subsect:appl} 
is stated as follows.

\begin{thm}\label{thm:fixed point free}
Let $R$, $\sigma $, $x_1$, $x_2$ and $f$ be as above, 
and assume that $f$ is in $k^*$. 
Then, 
$\sigma $ is exponential over $R$ 
if and only if 
there exists 
$\psi \in H(f)$ 
such that the $\Ga $-action 
$\widetilde{E}=
\begin{psmallmatrix}
\psi (x_1) & \psi (x_2) \\
\psi (x_1)+fT & \psi (x_2)
\end{psmallmatrix}$ 
on $k[X_1,X_2]$ restricts to $R[X_1,X_2]$. 
\end{thm}
\begin{proof}
Let $\psi \in H(f)$ be as in the theorem. 
Since 
$\psi \in H(f)\subset C(\tau )$ and $f\in k^*$, 
we have 
$\tau =\tau ^{\psi }=(x_1+f,x_2)^{\psi }=
\begin{psmallmatrix}
\psi (x_1) & \psi (x_2) \\
\psi (x_1)+f & \psi (x_2) \\
\end{psmallmatrix}=\widetilde{E}_1$ 
(cf.~Lemma~\ref{lem:rank 1} (iii)). 
Since $\widetilde{E}$ restricts to $R[X_1,X_2]$ by assumption, 
it follows that 
$\sigma $ is exponential over $R$.

Conversely, 
assume that $\sigma $ is exponential over $R$. 
We would like to 
show that there exists 
$\psi \in H(f)$ as in the theorem. 
Note that 
$\begin{psmallmatrix}
\psi (x_1+u_1)    & \psi (ax_2+u_2)\\
\psi (x_1+u_1)+fT & \psi (ax_2+u_2)
\end{psmallmatrix}
=\begin{psmallmatrix}
\psi (x_1)    & \psi (x_2)\\
\psi (x_1)+fT & \psi (x_2)
\end{psmallmatrix}$ 
for any $\psi \in \Aut _kk[X_1,X_2]$, 
$a\in k^*$ and $u_1,u_2\in k$. 
Hence, 
we may take $\psi $ from $H(f)H_0$, 
and thus from $C(\tau )$ thanks to Theorem~\ref{thm:C(tau)}.

By assumption, 
there exists 
a $\Ga $-action $F$ on $R[X_1,X_2]$ over $R$ with $F_1=\sigma $. 
The extension $\widetilde{F}$ of $F$ to $k[X_1,X_2]$ has rank one 
by Theorem~\ref{thm:RM}. 
We also have $\widetilde{F}_1=\tau =(x_1+f,x_2)$. 
Hence, 
$k[X_1,X_2]^{\widetilde{F}}$ satisfies (a) of Lemma~\ref{lem:rank 1}~(ii). 
Since $f$ is in $k^*$, 
this implies by Lemma~\ref{lem:rank 1} (iii) that 
$k[X_1,X_2]^{\widetilde{F}}=\psi (k[x_2])=k[p_2]$ 
for some $\psi \in C(\tau )$, 
where $p_i:=\psi (x_i)$. 
In this situation, 
the $\Ga $-action 
$\widetilde{E}:=\begin{psmallmatrix}
p_1 & p_2\\
p_1+(\widetilde{F}_1(p_1)-p_1)T & p_2
\end{psmallmatrix}$ 
on $k[X_1,X_2]$ 
restricts to $R[X_1,X_2]$ by Example~\ref{ex:modification}, 
since $R$ is a UFD by assumption. 
Moreover, 
$\widetilde{F}_1(p_1)-p_1=\tau (p_1)-p_1=f$ holds by (\ref{eq:C(ep)}), 
since $(p_1,p_2)\in C(\tau )$. 
Thus, 
$\widetilde{E}$ 
is equal to 
$\begin{psmallmatrix}
p_1 & p_2\\
p_1+fT & p_2
\end{psmallmatrix}$. 
Therefore, 
$\psi $ satisfies the required condition. 
\end{proof}

\section{$\Ga $-actions inducing an elementary automorphism}
\label{sect:set of actions}
\setcounter{equation}{0}

Let $k$ be a field. 
For $t\in k^*$, 
we set $\ep :=(x_1+t,x_2,\ldots ,x_n)\in \Aut _k\kx $. 
In this section, 
we investigate the set $\cA _k(\ep )$ mainly when $n\ge 3$. 
It turns out that for $r=1,\ldots ,n-1$, 
there exists $A\in \cA _k(\ep )$ with $\gamma (A)=n-r$, 
or equivalently, 
there exists a $\Ga $-action $E$ on $\kx $ of rank $r$ 
with $E_1=\ep $. 
We note that if $p=0$, 
then $E:=(x_1+tT,x_2,\ldots ,x_n)$ is the only $\Ga $-action on $\kx $ 
satisfying $E_1=\ep $.

Replacing $x_1$ with $t^{-1}x_1$, 
we may assume without loss of generality that $t=1$. 
So, 
let $\ep :=(x_1+1,x_2,\ldots ,x_n)$ 
in what follows. 
Note by Lemma~\ref{lem:invariant ring} that 
\begin{equation}\label{eq:k[x]^ep}
\kx ^{\ep }=k[x_1^p-x_1,x_2,\ldots ,x_n]. 
\end{equation}

\subsection{}\label{subsect:A(e)1}
By Lemma~\ref{lem:rank 1} (ii) and (iii), 
we have 
\begin{equation}\label{eq:A(e) rank 1}
\{ A\in \cA _k(\ep )\mid \gamma (A)=n-1\} 
=\{ k[f_2,\ldots ,f_n]\mid (f_1,\ldots ,f_n)\in C(\ep )\} . 
\end{equation}
If $n=2$, 
then $\cA _k(\ep )$ is equal to (\ref{eq:A(e) rank 1}) 
due to Theorem~\ref{thm:RM}. 
We have a complete description of it thanks to 
Theorem~\ref{thm:C(tau)}.

We remark that $\phi \in \Aut _k\kx $ belongs to $C(\ep )$ 
if and only if $\phi $ has the form $(x_1+v_1,v_2,\ldots ,v_n)$ 
for some $v_i\in \kx ^{\ep }$. 
Hence, 
$C(\ep )$ contains the subgroup $C_0(\ep )$ of $\T _n(k)$ generated by 
automorphisms 
$(x_1,\ldots ,x_{i-1},ax_i+g,x_{i+1},\ldots ,x_n)$, 
where 
$$
\left\{ \begin{array}{l}
\text{$a=1$ and $g\in k[x_2,\ldots ,x_n]$ if $i=1$},\\
\text{$a\in k^*$ and 
$g\in k[x_1^p-x_1,x_2,\ldots ,x_{i-1},x_{i+1},
\ldots ,x_n]$ if $i\ne 1$.} 
\end{array}\right.
$$
Theorem~\ref{thm:C(tau)} implies that 
$C(\ep )=C_0(\ep )$ when $n=2$. 
However, 
it is not clear if $C(\ep )=C_0(\ep )$ holds when $n\ge 3$.

For example, 
set $f:=x_2x_3+x_1-x_1^p$ 
and $R:=k(x_3)[x_4,\ldots ,x_n]$. 
Then, we have $R[x_1,x_2]=R[x_1,f]$. 
We claim that 
the $\Ga $-action $\widetilde{F}:=\begin{psmallmatrix}
x_1     &f \\
x_1+x_3T&f 
\end{psmallmatrix}$ 
on $R[x_1,x_2]$ restricts to a $\Ga $-action $F$ on $\kx $, 
since 
the equation $\widetilde{F}(f)=f$ implies 
$$
\widetilde{F}(x_2)
=x_2+x_3^{-1}(x_1-x_1^p-\widetilde{F}(x_1-x_1^p))
=x_2-T+x_3^{p-1}T^p\in \kx [T].
$$
Pick any 
$h\in k[f,x_3,\ldots ,x_n]\subset \kx ^F$. 
Then, 
we see from the above remark that 
\begin{equation}\label{eq:F_h}
F_h=(x_1+x_3h,x_2-h+x_3^{p-1}h^p,x_3,\ldots ,x_n)
\in C(\ep ). 
\end{equation}
When $n=3$, 
we do not know if $F_f$ belongs to $C_0(\ep )$. 
In fact, 
a famous conjecture of 
Nagata~\cite{Nagata} asserts that a similar type of automorphism of $\kx $ 
does not belong to $\T _3(k)$. 
On the other hand, $F_f$ belongs to $C_0(\ep )$ when $n=4$. 
In fact, 
$F_f=\tau F_{x_4}\tau ^{-1}F_{x_4}^{-1}$ 
holds with $\tau :=(x_1,x_2,x_3,x_4+f)\in C_0(\ep )$ (cf.~\cite{Smith}), 
since 
\begin{align*}
\begin{pmatrix}
x_1\! & \! x_4\\
x_1\! & \! x_4+f\\
\end{pmatrix}\!
\begin{pmatrix}
x_1\! & \! x_4\\
x_1+x_3x_4\! & \! x_4\\
\end{pmatrix}\!
\begin{pmatrix}
x_1\! & \! x_4\\
x_1\! & \! x_4-f\\
\end{pmatrix}\!
\begin{pmatrix}
x_1\! & \! x_4\\
x_1-x_3x_4\! & \! x_4\\
\end{pmatrix}
=
\begin{pmatrix}
x_1\! & \! x_4\\
x_1+x_3f\! & \! x_4
\end{pmatrix}
\end{align*}
holds in $\Aut _{k(f,x_3)}k(f,x_3)[x_1,x_4]$. 
We see from (\ref{eq:F_h}) 
that $F_{x_4}$ is also in $C_0(\ep )$.

\subsection{}\label{subsect:A(e)2}
Let $2\le r<n$. 
We construct a $\Ga $-action $E$ on $\kx $ of rank $r$ 
with $E_1=\ep $. 
Set $f_1:=x_1+x_n^{-1}x_r^p$,  
$$
f_i:=x_i+x_n^{-1}(x_2^p+x_3^p+\cdots +x_{i-1}^p)
+x_n^{p-1}(f_1^p-f_1)
\quad\text{for}\quad i=2,\ldots ,r
$$
and $f_i:=x_i$ for $i=r+1,\ldots ,n-1$. 
Then, 
we note the following: 

\nd (a) 
$B:=k[x_n^{\pm 1}][f_1,\ldots ,f_{n-1}]$ is contained in 
$B':=k[x_n^{\pm 1}][x_1,\ldots ,x_{n-1}]$. 

\nd (b) 
$x_i\in k[x_n^{\pm 1}][f_i,f_1,x_2,\ldots ,x_{i-1}]$ for $i=2,\ldots ,r$ 
and $x_1\in k[x_n^{-1},f_1,x_r]$. 
\label{se:2--r}

\nd 
Using (b), 
we can prove 
$x_i\in B$ for $i=2,\ldots ,r,1$ by induction on $i$. 
Therefore, 
we have $B=B'$. 
We define a $\Ga $-action 
$\widetilde{E}$ on $B'$ 
over $k[x_n^{\pm 1}]$ 
by $\widetilde{E}=\begin{psmallmatrix}
f_1     &f_2&\cdots &f_{n-1}\\
f_1+T   &f_2&\cdots &f_{n-1}
\end{psmallmatrix}$.

\begin{prop}\label{prop:rank n-1}
\nd{\rm (i)} 
$\widetilde{E}$ restricts to 
a $\Ga $-action $E$ on $\kx $ with $E_1=\ep $. 

\nd{\rm (ii)} 
$\kx ^E=k[x_nf_2,\ldots ,x_nf_r,x_{r+1},\ldots ,x_n]$, 
and $E$ has rank $r$. 

\end{prop}
\begin{proof}
(i) Since $\widetilde{E}(x_i)=x_i$ for $i=r+1,\ldots ,n$, 
it suffices to show the following (c):

\nd (c) 
$\widetilde{E}(x_i)-x_i$ for $i=2,\ldots ,r$ 
and $\widetilde{E}(x_1)-x_1-T$ 
belong to $I:=x_n(T^p-T)\kx [T]$.

For $i=2,\ldots ,r$, 
the equation $\widetilde{E}(f_i)=f_i$ gives 
$$
\widetilde{E}(x_i)-x_i=
-x_n^{-1}\sum _{j=2}^{i-1}(\widetilde{E}(x_j)-x_j)^p
-x_n^{p-1}(T^p-T). 
$$
Using this, 
we can prove $\widetilde{E}(x_i)-x_i\in I$ 
for $i=2,\ldots ,r$ 
by induction on $i$, 
since 
$\widetilde{E}(x_j)-x_j\in I$ implies 
$x_n^{-1}(\widetilde{E}(x_j)-x_j)^p\in I$. 
The equation $\widetilde{E}(f_1)=f_1+T$ gives 
$\widetilde{E}(x_1)-x_1-T=-x_n^{-1}(\widetilde{E}(x_r)-x_r)^p$, 
which belongs to $I$ similarly.

(ii) 
Set $\mathcal{F}:=\{ x_nf_2,\ldots ,x_nf_r,x_{r+1},\ldots ,x_{n-1}\} 
\subset \kx $. 
Then, since 
$(x_nf_i)|_{x_n=0}=\sum _{j=2}^{i-1}x_j^p+x_r^{p^2}$ for $i=2,\ldots ,r$, 
the image of $\mathcal{F}$ in 
$\kx /x_n\kx \simeq k[x_1,\ldots ,x_{n-1}]$ 
is algebraically independent over $k$. 
This implies that $k[\mathcal{F}]\cap x_n\kx =\zs $.

From the definition of $\widetilde{E}$, 
we see that 
$B^{\widetilde{E}}=
k[x_n^{\pm 1}][f_2,\ldots ,f_{n-1}]=k[\mathcal{F}][x_n^{\pm 1}]$. 
Hence, 
we have $\kx ^E=k[\mathcal{F}][x_n^{\pm 1}]\cap \kx $. 
We show that 
$k[\mathcal{F}][x_n^{\pm 1}]\cap \kx =k[\mathcal{F}][x_n]$ 
by contradiction. 
Supposing the contrary, 
there exist $l\ge 1$ and 
$g=\sum _{i\ge 0}a_ix_n^i\in k[\mathcal{F}][x_n]$ 
such that $x_n^{-l}g\in \kx $, 
where $a_i\in k[\mathcal{F}]$ and $a_0\ne 0$. 
Then, 
since $g\in x_n^l\kx $, 
we have $a_0=g-\sum _{i\ge 1}a_ix_n^i\in x_n\kx $, 
and so $a_0\in k[\mathcal{F}]\cap x_n\kx =\zs $, 
a contradiction. 
This proves the first part of (ii).

Since $\kx ^E$ 
contains $x_{r+1},\ldots ,x_n$, 
we have $\gamma (\kx ^E)\ge n-r$. 
On the other hand, 
since $x_nf_i$ has no linear part for $i=2,\ldots ,r$, 
we see that 
$(\kx ^E)^{\lin }=kx_{r+1}+\cdots +kx_n$. 
Hence, we have $\gamma (\kx ^E)\le \dim _k(\kx ^E)^{\lin }=n-r$ 
by Remark~\ref{rem:rank ineq} (ii). 
Thus, we get 
$\gamma (\kx ^E)=n-r$. 
Therefore, 
$E$ has rank $r$. 
\end{proof}

\subsection{}\label{subsect:A(e)3}
When $n\ge 3$, 
we do not know if there exists 
a $\Ga $-action $E$ on $\kx $ of rank $n$ with $E_1=\ep $ 
(Question~\ref{q:rank3}). 
In the following, we consider the case $n=3$.

At present, 
we know of no rank three 
$\Ga $-action $E$ on $\kx $ with $\gamma (\kx ^{E_1})>0$ 
in the first place, 
although $\gamma (\kx ^{\ep })=2$. 
In relation to this, 
the author obtained the following result 
(see Theorem 3.1 of \cite{ch order} for a more general statement).

\begin{thm}[\cite{ch order}]\label{thm:ch order}
Let $n=3$ and let $E$ be a rank three $\Ga $-action on $\kx $. 
If $f_1,f_2\in \kx ^E$ are algebraically independent over $k$, 
then we have $\gamma (\kx ^{E_{f_1f_2}})=0$. 
\end{thm}

However, 
there exists a rank three $\Ga $-action $E$ on $\kx $ 
satisfying 
$\kx ^E\subset \kx ^{\ep }$.

\begin{rem}\label{rem:BEBT}\rm 
Let $\tau \in \Aut B$ be of order $p$, 
and $E$ a $\Ga $-action on $B$ with $B^E\subset B^{\tau }$. 
Pick any $(S,r)\in \cP _E$ (cf.~\S \ref{sect:modification}). 
Then, 
$\tau $ extends to an element 
$\widetilde{\tau }$ of $\Aut _{B_S^E}B_S=\Aut _{B_S^E}B_S^E[r]$, 
since $S\subset B^E\subset B^{\tau }$. 
Moreover, we have

\nd (i) 
$a:=\widetilde{\tau }(r)-r\in B_S^E$. 
Actually, we may regard $\widetilde{\tau }\in 
\Aut _{B_S^E}B_S^E[r]=\J _1(B_S^E)$, 
so $\widetilde{\tau }(r)$ is in $r+B_S^E$ 
by Remark~\ref{rem:strict triangular} (i).   
Hence, 
$\tau $ is a quasi-translation of $B$ over $B^E$.

\nd (ii) The $\Ga $-action 
$E':B_S^E[r]\ni h(r)\mapsto h(r+aT)\in B_S^E[r][T]$ 
satisfies $E'_1=\widetilde{\tau }$, 
but $E'$ does not necessarily restrict to $B$ 
even when $\tau $ is exponential. 
\end{rem}

Let us construct a rank three $\Ga $-action on $\kx $ mentioned above. 
Set 
\begin{equation}\label{eq:f,r,g}
f:=x_1^{p^2}-x_1^p+x_2x_3,\quad 
g:=f^{p^2}x_3-x_2^{p^2-1}+f^{p^2-p}x_2^{p-1}
\quad\text{and}\quad
r:=fx_1+x_2.
\end{equation}
Then, we have 
$k[f,g]\subset \kx ^{\ep }$ by (\ref{eq:k[x]^ep}), 
and $\gamma (k[f,g])=0$ by Remark~\ref{rem:rank ineq} (ii).

Now, set $A:=k[f,g,(fg)^{-1}]$. 
First, we show that $\kx \subset A[r]$. 
Observe that
\begin{equation}\label{eq:rank3 xi}
\begin{aligned}
\xi &:=f^{p^2+1}-r^{p^2}+f^{p^2-p}r^p \\
&=\underline{f^{p^2}(x_1^{p^2}}-\underline{\underline{x_1^p}}+x_2x_3)
-(\underline{f^{p^2}x_1^{p^2}}+x_2^{p^2})
+\underline{\underline{f^{p^2-p}(f^px_1^p}}+x_2^p)
=gx_2. 
\end{aligned}
\end{equation} 
Hence, 
we have 
$x_2=g^{-1}\xi \in A[r]$. 
Then, $x_1,x_3\in A[r]$ follow from 
\begin{equation}\label{re:x_1,x_3}
x_1=f^{-1}(r-x_2)\quad \text{and}\quad 
x_3=f^{-p^2}(g+x_2^{p^2-1}-f^{p^2-p}x_2^{p-1}).
\end{equation}

From $\kx \subset A[r]$, 
it follows that $f$, $g$ and $r$ are algebraically independent over $k$. 
Thus, 
we have $A[r]=A^{[1]}$. 
For $l,m\ge 0$, 
we define a $\Ga $-action $\widetilde{E}^{l,m}$ on $A[r]$ by 
$$
\widetilde{E}^{l,m}:A[r]\ni h(r)\mapsto h(r+f^lg^mT)\in A[r][T]. 
$$

\begin{prop}\label{prop:rank 3}
In the notation above, 
the following assertions hold.

\nd{\rm (i)} 
$\widetilde{E}_1^{l,m}$ restricts to $\kx $ 
if and only if $l,m\ge 1$ or $(l,m)=(1,0)$. 
Moreover, 
$\ep $ is the restriction of $\widetilde{E}_1^{1,0}$ to $\kx $.

\nd{\rm (ii)} 
$\widetilde{E}^{l,m}$ restricts to $\kx $ 
if and only if $l,m\ge 1$.

\nd{\rm (iii)} 
Let $E^{l,m}$ be the restriction of 
$\widetilde{E}^{l,m}$ to $\kx $, 
where $l,m\ge 1$. 
Then, we have $\kx ^{E^{l,m}}=k[f,g]$. 
Hence, 
$E^{l,m}$ has rank three 
and satisfies $\kx ^{E^{l,m}}\subset \kx ^{\ep }$. 
\end{prop}
\begin{proof}
Since $k[f,g]\subset \kx ^{\ep }$ 
and 
$\ep (r)=f\ep (x_1)+x_2=r+f=\widetilde{E}_1^{1,0}(r)$, 
the extension of $\ep $ to $A[r]$ is equal to $\widetilde{E}_1^{1,0}$. 
This proves the last part of (i). 
For the rest of (i) and (ii), 
it suffices to verify the following (a) and (b):

\nd (a) 
$\widetilde{E}:=\widetilde{E}^{1,1}$ restricts to $\kx $.

\nd (b) 
$\widetilde{E}_1^{l,0}(x_2)\not\in \kx $ 
if $l\ge 2$, 
$\widetilde{E}^{1,0}(x_2)\not\in \kx [T]$, 
and 
$\widetilde{E}_1^{0,m}(x_1)\not\in \kx $ if $m\ge 0$.

Actually, (a) implies the ``if" part of (ii) 
in view of Remark~\ref{rem:multiple of Ga-action}. 
Moreover, the ``if" part of (ii) implies that 
$\widetilde{E}_1^{l,m}$ restricts to $\kx $ if $l,m\ge 1$. 
Similarly, 
(b) implies the ``only if" parts of (i) and (ii).

First, we show (a). 
Since $\widetilde{E}$ fixes $f$ and $g$, 
(\ref{eq:rank3 xi}) gives that 
\begin{equation}\label{align:rank3}
\begin{aligned}
&\widetilde{E}(x_2)=\widetilde{E}(g^{-1}\xi )
=g^{-1}(\underline{f^{p^2+1}-(r}+fgT)^{p^2}+\underline{f^{p^2-p}(r}+fgT)^p) \\
&\quad 
=g^{-1}(\underline{\xi }-(fgT)^{p^2}+f^{p^2-p}(fgT)^p) 
=x_2-f^{p^2}g^{-1}((gT)^{p^2}-(gT)^p). 
\end{aligned}
\end{equation}
Hence, 
$\widetilde{E}(x_2)$ belongs to $\kx [T]$. 
From (\ref{re:x_1,x_3}) and (\ref{align:rank3}), 
we obtain 
\begin{equation}\label{align:rank3 2}
\begin{aligned}
\widetilde{E}(x_1)
=\widetilde{E}(f^{-1}(r-x_2))
&=\underline{f^{-1}(r}+fgT\,\underline{-\,x_2}
+f^{p^2}g^{-1}((gT)^{p^2}-(gT)^p)) \\
&=\underline{x_1}+gT+f^{p^2-1}g^{-1}((gT)^{p^2}-(gT)^p)
\in \kx [T]. 
\end{aligned}
\end{equation}
Similarly, 
we can check that 
$\widetilde{E}(x_3)=
\widetilde{E}(f^{-p^2}(g+x_2^{p^2-1}-f^{p^2-p}x_2^{p-1}))$ 
belongs to $\kx [T]$ 
by noting $\widetilde{E}(x_2)\in x_2+f^{p^2}\kx [T]$. 
This proves (a).

Next, we show (b). 
For $i=1,2$, 
let $\pi _i:\kx [T]\to k[x_i][T]$ be the substitution map 
defined by $x_j\mapsto 0$ for $j\ne i$. 
Then, 
we have $\pi _1(f)=x_1^{p^2}-x_1^p$ and $\pi _1(g)=0$. 
Now, 
suppose that $\widetilde{E}_1^{l,0}(x_2)\in \kx $ for some $l\ge 2$. 
Then, 
$\pi _1(g\widetilde{E}_1^{l,0}(x_2))$ is clearly zero. 
We note that 
$\widetilde{E}_1^{l,0}(x_2)$ is obtained from $\widetilde{E}(x_2)$ 
by replacing $T$ with $f^{l-1}g^{-1}$. 
Hence, 
we see from (\ref{align:rank3}) that 
$\pi _1(g\widetilde{E}_1^{l,0}(x_2))
=-\pi _1(f^{p^2}(f^{(l-1)p^2}-f^{(l-1)p}))\ne 0$. 
This is a contradiction. 
We can verify $\widetilde{E}^{1,0}(x_2)\not\in \kx [T]$ similarly 
by noting that $\widetilde{E}^{1,0}(x_2)=\widetilde{E}(x_2)|_{T=g^{-1}T}$. 
Since $\widetilde{E}_1^{0,m}(x_1)
=\widetilde{E}(x_1)|_{T=f^{-1}g^{m-1}}$ 
and since $\pi _2(f)=0$ and $\pi _2(g)=-x_2^{p^2-1}$, 
we see from (\ref{align:rank3 2}) that 
$\pi _2(fg\widetilde{E}_1^{0,m}(x_1))=\pi _2(g^{m+1}+g^{mp^2})\ne 0$ 
for any $m\ge 0$. 
This shows that $\widetilde{E}_1^{0,m}(x_1)\not\in \kx [T]$ as before.

For (iii), 
note that $\kx ^{E^{l,m}}=A\cap \kx $, 
since $A[r]^{\widetilde{E}^{l,m}}=A$ 
(cf.~Example~\ref{ex:Ga-action on R[x]}). 
We can prove that $A\cap \kx =k[f,g]$ 
in the same manner 
as Theorem 5.1 of \cite{ch order}, 
where the roles of $x_1$ and $x_2$ are interchanged. 
In fact, 
the proof of this theorem shows that 
$k[\hat{f},\hat{g},(\hat{f}\hat{g})^{-1}]\cap \kx =k[\hat{f},\hat{g}]$ 
holds for any $\hat{f},\hat{g}\in \kx \sm \zs $ 
with $\hat{f}|_{x_3=0}\in k[x_2]\sm k$, 
$\hat{g}|_{x_3=0}\not\in k[x_2]$, 
$\hat{f}|_{x_2,x_3=0}=0$ and $\hat{g}|_{x_2,x_3=0}\not\in k$. 
The last part of (iii) follows from the remark after (\ref{eq:f,r,g}). 
\end{proof}

We mention that, for a certain family of rank three $\Ga $-actions on $\kx $, 
the invariant rings 
for the induced automorphisms were studied in detail in \cite{ch order}. 
See also \cite{LSC} for a construction of rank three 
$\Ga $-actions on $\kx $ for $p=0$.

\section{Questions}\label{sect:q and r}
\setcounter{equation}{0}

\subsection{}
Recall the sequence of subgroups 
$\T _n(R)\subset \T '_n(R)\subset \T ''_n(R)\subset \Aut _R\Rx $ 
discussed in Section~\ref{sect:intro}. 
This paper is strongly motivated by 
the following question 
(see also \cite[Question 7.5]{ch order}).

\begin{q}\label{q:exp gen}\rm 
Assume that $n\ge 3$, 
or $n=2$ and $R$ is not a field. 

\nd (1) Does $\T _n'(R)=\T _n''(R)$ hold?

\nd (1) Does $\T _n''(R)=\Aut _R\Rx $ hold? 

\end{q}

\subsection{}
For $\sigma \in \Aut B$ and $r\ge 0$, 
we denote by $\sigma ^{[r]}\in \Aut B[y_1,\ldots ,y_r]$ 
the extension of $\sigma $ defined by 
$\sigma ^{[r]}(y_i)=y_i$ for $i=1,\ldots ,r$, 
where $y_1,\ldots ,y_r$ are variables. 
We say that $\sigma \in \Aut B$ is {\it stably exponential} (over $R$) 
if $\sigma ^{[r]}$ is exponential (over $R$) for some $r\ge 0$. 
Clearly, 
``exponential" implies ``stably exponential". 
It is easy to see that no elements of $\Aut _R\Rxx \sm \{ \id \} $ 
are stably exponential because of Remark~\ref{rem:local} (ii) 
(see also Remark~\ref{rem:quasi-translation} (ii)).

\begin{q}\label{q:stb exp1}\rm 
Does there exist an element of $\Aut _R\Rx $ of order $p$ 
which is not stably exponential over $R$?
\end{q}

The {\it Stable Tameness Conjecture} 
asserts that any 
$\sigma \in \Aut _R\Rx $ is {\it stably tame}, 
i.e., 
$\sigma ^{[r]}\in \T _{n+r}(R)$ holds 
for some $r\ge 0$ (cf.~\cite[Conjecture 6.1.8]{Essen}). 
The conjecture is known to be true if $n=2$ and $R$ is a regular ring 
by Berson-van den Essen-Wright~\cite[Theorem 4.10]{BEW}. 
This implies that 
the generic elementary automorphism 
$\sigma ^g$ given in Section~\ref{subsect:non exp aut} 
is stably tame if $R$ is regular, 
since $\sigma ^g$ lies $\Aut _{R[\z ]}R[\z ][x,y]$, 
and $R[\z ]$ is regular if so is $R$.

We can generalize the notion of ``stable tameness" in two ways.

\begin{q}\label{q:stb exp gen}\rm 
Let $\sigma \in \Aut _R\Rx $.

\nd (i) Does there always exist $r\ge 0$ 
such that $\sigma ^{[r]}$ belong to $\T''_{n+r}(R)$?

\nd (ii) Does there always exist $r\ge 0$ 
such that $\sigma ^{[r]}$ belong to $\T'_{n+r}(R)$? 

\end{q}

\subsection{}
Theorems~\ref{thm:main} and \ref{thm:non-exp} 
lead to the following four questions.

\begin{q}\label{q:over a field}\rm 
Let $k$ be a field. 
Does there exist an element of $\Aut _k\kx $ of order $p$ 
which is not exponential over $k$?
\end{q}

\begin{q}\label{q:fixed point free}\rm 
Does there exist a fixed point free 
generic elementary automorphism of $\Rx $ 
(see the beginning of Section~\ref{sect:center}) 
which is not exponential over $R$?
\end{q}

\begin{q}\label{q:higher dim}\rm 
Assume that $n\ge 4$, 
or $n=3$ and $R$ is not a field. 
Does there exist an element of $\J _n(R)$ of order $p$ 
which is not exponential over $R$?
\end{q}

More specifically, 
we may ask if $E_1,E_1'\in \J _3(R)$ 
in Example~\ref{ex:triangular} are exponential. 
By construction, 
they are fixed point free 
generic elementary automorphisms.

\begin{q}\label{q:generic elementary}\rm 
Let $\sigma ^g$ be the 
generic elementary automorphism of $\Rxyz $ 
given in Section~\ref{subsect:non exp aut} 
(or one constructed in a similar manner). 

\nd {\rm (i)} 
Is $\sigma ^g$ exponential over $R$ 
in the case where $k[\widetilde{y},\z ]$ is not $(1+ug)$-stable over $k$?

\nd{\rm (ii)} Is $\sigma ^g$ stably exponential over $R$?

\nd {\rm (iii)} 
Is $\sigma ^g$ exponential over a field contained in $R$?

\nd{\rm (iv)} Does $\sigma ^g$ belong to $\T '_{2+l}(R)$? 

\end{q}

\subsection{}
We call $\sigma \in \Aut B$ a {\it semi-translation} of $B$ 
if there exists a ring $\mathcal{K}$ and $a\in \mathcal{K}\sm \zs $ 
such that $B$ can be embedded into the polynomial ring $\mathcal{K}[x]$ and 
$\sigma $ is the restriction of $\tau \in \Aut _{\cK }\mathcal{K}[x]$ 
defined by $\tau (x)=x+a$. 
Note that $\mathcal{K}[x]\subset Q(\mathcal{K})[a^{-1}x]$ 
and $\tau (a^{-1}x)=a^{-1}x+1$, 
so we may assume that $\mathcal{K}$ is a field and $a=1$ 
in this definition. 
Clearly, 
a quasi-translation is a semi-translation, 
and a non-identity semi-trans\-la\-tion has order $p$. 
Remark~\ref{rem:quasi-translation} also applies to semi-translations.

\begin{q}\label{q:converse of C}\rm 
Let $\sigma \in \Aut _R\Rx $ be of order $p$.

\nd (1) Is $\sigma $ always a quasi-translation of $\Rx $? 

\nd (2) Is $\sigma $ always a semi-translation of $\Rx $? 

\end{q}

Let us call a semi-translation of $B$ 
a {\it super-translation} of $B$ 
if it is not a quasi-translation of $B$. 
Every super-translation has order $p$, 
but is not exponential by Remark~\ref{rem:local} (v).

\begin{q}\label{q:super-translation}\rm 
Does there exist a super-translation of $\Rx $?
\end{q}

\subsection{}
The following question is also of interest (cf.~Section~\ref{subsect:A(e)3}).

\begin{q}\label{q:rank3}\rm 
Assume that $n=3$ and $k$ is a field. 
Does there exist a rank three $\Ga $-action $E$ on $\kx $ 
such that $E_1=(x_1+1,x_2,x_3)$? 
\end{q}

\begin{thebibliography}{AMS}

\bibitem{BEW}
J. Berson, A. van den Essen\ and\ D. Wright, 
Stable tameness of two-dimensional polynomial automorphisms 
over a regular ring, Adv. Math. {\bf 230} (2012), no.~4-6, 2176--2197.



\bibitem{ML}
A.\ J.\ Crachiola and L.\ G.\ Makar-Limanov, 
Exponential maps of a polynomial ring in two variables, 
S\~ao Paulo J.\ Math.\ Sci.\ {\bf 13} (2019), 73--82. 



\bibitem{EH}
M.\ El Kahoui and A.\ Hammi, 
On exponential morphisms over commutative rings, 
Colloq.\ Math.\ {\bf 169} (2022), 333--340. 



\bibitem{Essen}
A. van den Essen, 
{\it Polynomial automorphisms and the Jacobian conjecture}, 
Progress in Mathematics, 190, Birkh\"auser Verlag, Basel, 
2000. 



\bibitem{LSC}
G.~Freudenburg, 
Local slice constructions in $k[X,Y,Z]$, 
Osaka J.\ Math.\ {\bf 34} (1997), 757--767. 


\bibitem{Ig}T. Igarashi, 
Finite subgroups of the automorphism group of the affine plane, 
M. A. thesis, Osaka University, 1977. 


\bibitem{Jung}H.~Jung, 
\"Uber ganze birationale Transformationen der Ebene, 
J.\ Reine Angew.\ Math.\ {\bf 184} (1942), 161--174. 


\bibitem{Kojima}H. Kojima, 
Locally finite iterative higher derivations on $k[x,y]$, 
Colloq. Math. {\bf 137} (2014), no.~2, 215--220. 


\bibitem{Kulk}W.~van der Kulk, 
On polynomial rings in two variables, 
Nieuw Arch.\ Wisk. (3) {\bf 1} (1953), 33--41. 


\bibitem{Nihonkai}
S. Kuroda, 
The automorphism theorem and 
additive group actions on the affine plane, 
Nihonkai Math.\ J.\ {\bf 28} (2017) 65--68. 


\bibitem{ch order}S. Kuroda, 
Polynomial automorphisms of characteristic order and their invariant rings, 
Transform. Groups (2022). 



\bibitem{Maubach}
S. Maubach, 
Invariants and conjugacy classes of triangular polynomial maps, 
J. Pure Appl. Algebra {\bf 219} (2015), no.~12, 5206--5224.


\bibitem{MiyanishiNagoya}
M.~Miyanishi, 
$G\sb{a}$-action of the affine plane, 
Nagoya Math. J. {\bf 41} (1971), 97--100. 



\bibitem{MiyanishiTata}M.~Miyanishi, 
{\it Curves on rational and unirational surfaces}, 
Tata Institute of Fundamental Research Lectures 
on Mathematics and Physics, 60, 
Tata Inst. Fund. Res., Bombay, 1978. 


\bibitem{Miyanishi3}
M. Miyanishi, 
Wild $\Bbb{Z}/p\Bbb{Z}$-actions on algebraic surfaces, 
J. Algebra {\bf 477} (2017), 360--389. 



\bibitem{MI}
M. Miyanishi\ and\ H. Ito, 
{\it Algebraic surfaces in positive characteristics---purely 
inseparable phenomena in curves and surfaces}, 
World Scientific Publishing Co. Pte. Ltd., 
Hackensack, NJ, 2021. 



\bibitem{Nagata}
M. Nagata, {\it On automorphism group of $k[x,\,y]$}, 
Kinokuniya Book Store Co., Ltd., Tokyo, 1972. 


\bibitem{Rentschler}
R. Rentschler, 
Op\'erations du groupe additif sur le plan affine, 
C. R. Acad. Sci. Paris S\'er. A-B {\bf 267} (1968), 
384--387. 


\bibitem{tree}
J.-P. Serre, {\it Trees}, 
translated from the French by John Stillwell, 
Springer, Berlin, 1980. 



\bibitem{SU}
I.~Shestakov and U.~Umirbaev, 
The tame and the wild automorphisms of polynomial rings in three variables, 
J.\ Amer.\ Math.\ Soc.\ {\bf 17} (2004), 197--227. 


\bibitem{Smith}M. K. Smith, 
Stably tame automorphisms, 
J. Pure Appl. Algebra {\bf 58} (1989), 
209--212. 

\bibitem{Takeda}
Y. Takeda, Artin-Schreier coverings of algebraic surfaces, 
J. Math. Soc. Japan {\bf 41} (1989), no.~3, 415--435. 


\bibitem{Tani}
R. Tanimoto, Pseudo-derivations and modular invariant theory, 
Transform. Groups {\bf 23} (2018), no.~1, 271--297.

\bibitem{Tani2}
R.\ Tanimoto, 
On $p$-unipotent triangular automorphisms of polynomial 
rings in positive characteristic $p$, 
Saitama Math.\ J.\ {\bf 33} (2021), 1--11. 


\bibitem{Tani3}
R.\ Tanimoto, 
Triangular $\mathbb{Z}/3\mathbb{Z}$-actions on the affine 
four-space in characteristic three, 
Saitama Math.\ J.\ {\bf 35} (2023), 1--26. 



\bibitem{Wright}
D. Wright, 
Abelian subgroups of ${\rm Aut}\sb{k}(k[X,\,Y])$ 
and applications to actions on the affine plane, 
Illinois J. Math. {\bf 23} (1979), 579--634. 

\end{thebibliography}
\end{document}